\documentclass{amsart}
\usepackage{times,amssymb}
\usepackage{amsmath}
\usepackage{graphicx}
\usepackage{wrapfig}
\usepackage{natbib}
\usepackage{color}
\usepackage{appendix}

\def\RR{\mathbb{R}}

\def\R{\mathbb{R}}
\def\argmin{\text{argmin}}

\def\R{\mathbb{R}}

\def\sym{\text{sym}}
\def\t{\text}
\def\arc{\text{arc}}
\def\trace{\text{Trace}}

\newtheorem{theorem}{Theorem}[section]

\newtheorem{proposition}[theorem]{Proposition}
\newtheorem{corollary}[theorem]{Corollary}

\usepackage{url}

\newtheorem{remark}[theorem]{Remark}
\newtheorem{example}[theorem]{Example}

\begin{document}

\title[Differential Geometry  for Model Independent Analysis]{Differential Geometry  for Model Independent Analysis of Images and Other Non-Euclidean Data: Recent Developments}

\author{Rabi Bhattacharya and Lizhen Lin}

\dedicatory{ In celebration of Chuck's 70th Birthday}

\address{Department of Mathematics, The University of Arizona, Tucson, AZ}
\email{rabi@math.arizona.edu}

\address{Department of Applied and Computational Mathematics and Statistics, The University of Notre Dame, Notre Dame, IN}
\email{lizhen.lin@nd.edu}

\begin{abstract}

 This article provides an exposition of recent methodologies for nonparametric analysis of digital observations on images and other non-Euclidean objects. Fr\'echet means of distributions on metric spaces, such as manifolds and stratified spaces, have played an  important role in this endeavor. Apart from theoretical issues of uniqueness of the Fr\'echet  minimizer and the asymptotic distribution of the sample Fr\'echet  mean under uniqueness, applications to image analysis are highlighted.  In addition, nonparametric Bayes theory is brought to bear on the problems of density estimation and classification on manifolds.

\end{abstract}

\maketitle

\section{Introduction}

 Historically, directional statistics, that is, statistics on spheres, especially $S^2$, has been around for a long time, and there is a great deal of literature on it (See the books by \cite{watson83}, \cite{maridajpuu}, \cite{fisher87}). Much of that was inspired by a seminal paper by \cite{fisherra53} proving beyond any reasonable doubt that the earth's magnetic poles had shifted over geological times. Indeed, the two sets of data that he analyzed, one from the Quaternary period and the other from recent times (1947-48), showed an almost reversal of the directions of the magnetic poles. In addition to this first scientific demonstration of a phenomenon conjectured by some paleontologists, such studies of magnetic poles in fossilized remanent magnetism had an enormous impact on tectonics, essentially validating the theory of continental drift (\cite{Irvingbook}, \cite{fisher87}). There are other important applications of directional statistics, such as designing of windmills based on wind directions,  etc. Fisher's example is presented in Section \ref{sec-ex}, in comparison with the nonparametric method highlighted in this article.

The advancement of imaging technology and increase in computing prowess have opened up a whole new vista of applications. Medical imaging, for example, is now an essential component of medical practice. Not only have MRIs (magnetic resonance imaging) become routine for diagnosing a plethora of diseases, there are more advanced techniques such as the DTI (diffusion tensor imaging) which measures diffusion coefficients of water molecules in tiny voxels along nerve fibers in the cortex of the brain in order to understand or monitor diseases such as Parkinson's and Alzheimer's \citep{dtigood,dtikind,dtimor}. Beyond medicine, there are numerous applications to morphometrics \citep{bf},  graphics, robotics, and machine vision \citep{1333732,ma2005invitation,1524983}.

Images are geometric objects and their precise mathematical descriptions and identifications in different fields of applications are facilitated by the use of differential geometry.  \cite{kendall} and \cite{bf} were two pioneers in the geometric description and statistical analysis of images represented by landmarks on two or three dimensional objects. The spaces of such images, or shapes, are differential manifolds, or \emph{stratified spaces} obtained by gluing together manifolds of different dimensions.  In the following sections these spaces are described in detail. Much of the earlier statistical analysis on differential manifolds were parametric in nature, where a distribution $Q$ on a manifold $M$ is assumed to belong to a finite dimensional parametric family; that is, $Q$ is assumed to have a density (with respect a standard distribution, e.g., the volume measure on $M$) which is specified except for the value of a finite dimensional parameter $\theta$ lying in an open subset $\Theta$ of an Euclidean space.  The statistician'ÂÂs task is then to estimate the parameter (or test for its belonging to a particular subset of $\Theta$), using observed data. There are standard methodologies for estimation (say, the maximum likelihood estimator, MLE), or testing (such as the likelihood ratio test) that one may try to use. Of course, it still requires a great deal of effort to analytically compute these statistical indices and their (approximate) distributions on specific manifolds. A reasonably comprehensive account of these for the shape spaces of Kendall, or similar manifolds, may be found in \cite{dimk}.

The focus of the present article is a model independent, or nonparametric, methodology for inference on general manifolds. As a motivation consider the problem of discriminating between two distributions on an Euclidean space based on independent samples from them. In parametric inference one would use a density (with respect to a sigma-finite measure) which is specified except for a finite dimensional parameter as described above.  One may use one of a number of standard asymptotically efficient procedures to test if the two distributions have different parameter values (See, e.g., \cite{hotelling1931},\cite{goodall}). If the statistician is not confident about this parametric model, or any other, one popular method is to test for the differences between the means of the two distributions by using the two sample means.  When the sample sizes are reasonably large then the difference between the sample means is asymptotically normal with mean given by the difference between the population means. If the observations are from a normal distribution with the mean as the unknown parameter then this test is optimal in an appropriate sense (\cite{textbook}, pp 296-300, \cite{lehmann1959testing}, pp. 93,94).  But used in other parametric model the test is not, in general, optimal and may even be inconsistent; that is, there may be many pairs of distributions $Q_1\neq Q_2$ whose means are the same. However, when the components or coordinates of the distributions are such that the differences between $Q_1$ and $Q_2$ are reasonably expected to manifest in shifts of the mean vector, this widely used nonparametric test is quite effective, especially since  with large sample sizes the asymptotic distribution is normal. Turning now to distributions $Q$ on non-Euclidean metric spaces $S$, one has an analogue of the mean given by the minimizer, if unique, of the average (with respect to $Q$) of the squared distance from a point. This is the so called \emph{Fr\'echet mean} introduced by \cite{frech}, although physicists probably had used the notion earlier in specific physical contexts for the distribution $Q$ of  the mass of a body, calling it the center of mass.  Of course it is in general a non-trivial matter to find out broad conditions for the \emph{uniqueness of the Fr\'echet minimizer} and, in the case of uniqueness, to derive the \emph{(asymptotic) distribution of the sample Fr\'echet mean}. These allow one to obtain proper confidence regions for the Fr\'echet  mean of $Q$ and critical regions for tests for detecting differences in means of distributions on $M$ \citep{rabivic02, rabi03, rabi05}.  The theory of Fr\'echet means is presented in Section 2 (uniqueness and consistency), and in Section 4 (asymptotic distributions). The main results in Sections 2 and 4 are presented with complete proofs. Section 4 plays a central role for inference in the present context, and it contains some improvements of earlier results.

It has been shown in data examples that the nonparametric procedures based on Fr\'echet means often greatly outperform their parametric counterparts (See \cite{rabibook}).  Misspecification of the model is a serious issue with parametric inference, especially for distributions on rather complex non-Euclidean spaces.

In this article two types of images and their analysis are distinguished.  The greater emphasis is on landmarks based shapes introduced by \cite{kendall} and \cite{bf}.  This looks at a \emph{$k$-ad} or a set of $k$ properly chosen points, not all the same, on an $m$-dimensional image (usually $m=2$ or 3), $k>m$, such as an MRI scan of a section of the brain for purposes of diagnosing a disease, or a scan of some organ of a species for purposes of morphometrics. In order to properly compare images taken from different distances and angles using perhaps different machines, the \emph{shape of a $k$-ad} is defined modulo translation, scaling and rotation. The resulting shapes comprise \emph{Kendall's shape spaces.} In addition, one may consider \emph{affine shapes} which are invariant under all affine transformations appropriate in scene recognition; similarly, \emph{projective shapes} invariant under projective transformations are often used for robotic vision. The precise mathematical (geometric) descriptions of these kind of images are presented in Section 3. Sections 5 and 6 provide the asymptotic theory of tests and confidence regions on manifolds, based on the asymptotic distribution theory developed in Section 4. 

Section 8 considers briefly the second type of images, namely, the actual geometric shape of a compact two-dimensional surface or a three dimensional body. Here the shape space is infinite dimensional and may be viewed as a \emph{Hilbert manifold} \citep{Ellingson2013317}.  For purposes of diagnostics such as described above, this is probably not to be preferred in comparison with the finite dimensional landmarks based shapes considered by Kendall, because of the curse of dimensionality.  The Hilbert manifolds then are better suited for purposes of machine vision.  However, for that task a more effective methodology seems to be one which builds on the exciting inquiry of \cite{kac66}: Can one hear the shape of a drum? It turns out that for two-dimensional compact Riemannian manifolds such as compact surfaces, the spectrum of the Laplace Beltrami operator identifies the manifold in most cases, although there are exceptions. In three and higher dimensions, on the other hand, iso-spectral manifolds are not so rare \citep{Milnor01041964, Gordon1992, Zelditch2000}. Still, computer scientists and other researchers in machine vision have successfully implemented algorithms to identify two and three-dimensional images by the spectrum of their Laplaceans, sometimes augmented by their eigen-functions \citep{demmel09, Gotsman2003, Jain2007, shamir06, Reuter2009}. A mathematical breakthrough was achieved by \cite{Jones08}, who proved that indeed compact manifolds are determined by this augmentation.

 Section 7 is devoted to another very important statistical problem: \emph{nonparametric classification via density estimation, and nonparametric regression} on manifolds. In particular, we emphasize Ferguson's nonparametric Bayes theory of using \emph{Dirichlet process priors} for this endeavor \citep{ferg73, ferg74}.
 
 Sections \ref{sec-ex} provides a number of applications of the theory of Fr\'echet means, including Fisher's example mentioned above, but focusing on two-sample problems on landmarks based shape spaces such as those introduced by  Kendall  \citep{kendall,kendall89}.

The appendix, Section 10, provides a ready access to some notions in Riemannian geometry used in the text. 

\section{Existence of the Fr\'echet Mean on Non-Euclidean Spaces. }

Let $(S, \rho)$ be a metric space and $Q$ a probability measure on it.  The \emph{Fr\'echet function} of $Q$
 is defined as
 \begin{align}
 \label{eq-frefunc}
F(p) = \int \rho^2(p,q)Q(dq),   p \in S.
 \end{align}
If $F$ is finite at some $p$ then it is finite on $S$.  The set $C(Q)$ of minimizers of F is called the Fr\'echet  mean set. If the minimizer is unique, i.e., $C(Q)$ is a singleton, then it is called the \emph{Fr\'echet  mean} of $Q$, and one says that the \emph{Fr\'echet mean} of $Q$ exists. We will often use the topological condition
     \begin{align}
 \label{eq-compact}
 \text{\emph{ All closed bounded subsets of S are compact.    }}
 \end{align}
When $S$ is a Riemannian manifold and $\rho=\rho_g$ is the geodesic distance on it, then \eqref{eq-compact} is equivalent to the completeness of $S$, by the Hopf-Rinow theorem (\cite{docarmo}, pp. 146-149).

Let $X_1,\ldots,X_n $ be a random sample from $Q$, i.e., $X_j$ are i.i.d. with common distribution $Q$, defined on a probability space $(\Omega,\mathcal F, P)$. Denote by $F_n$ the Fr\'echet  function of the empirical $Q_n =(1/n) \sum_{1\leq j\leq n} \delta_{X_j}$, where $\delta_x$ is the point mass at $x.$ Also let $B^{\epsilon} =\{p\in S: \rho(p,B) < \epsilon\}$ for $B\subset S.$

\begin{theorem}[\citep{rabi03}]
\label{th-2.1}
  Assume \eqref{eq-compact} and that the Fr\'echet function $F$ of $Q$ is finite. Then (a) $C(Q)$ is nonempty and compact, and (b) for each $\epsilon>0$, there exists a random positive integer $N= N(\omega;\epsilon)$ and a $ P$-null set $\Gamma$ such that $\forall n\geq N(\omega;\epsilon)$,
                                   \begin{align}
                                   \label{eq-2.3}
                                   C(Q_n) \subset (C(Q))^{\epsilon} \;\text{for every}\; \omega\not\in \Gamma.
                                   \end{align}
(c) In particular, if the Fr\'echet mean of $Q$, say $\mu$, exists, then every measurable selection $\mu_n$ from $C(Q_n)$, converges almost surely to $\mu$. In this case $\mu_n$ is called the sample Fr\'echet mean.
\end{theorem}

\begin{proof}

First assume $S$ is compact.  Then (a) is obvious. To prove (b), it is enough to show that $\delta_n= \max\{\mid F_n(p) -F(p)\mid: p \in S\} \rightarrow 0 $ almost surely as $n\rightarrow \infty$. To see this let $\lambda= \min \{F(p): p \in S\}= F(q)$ $\forall q\in C(Q)$. If $(C(Q))^{\epsilon} = S$, then \eqref{eq-2.3} holds with $N=1$ (for every $\omega$). Assume $(C(Q))^{\epsilon}$ is not $S$, and write $M_1 = S\backslash(C(Q))^{\epsilon}$ . There exists $\theta(\epsilon) >0$, such that $\min\{F(p): p \in  M_1\} = \lambda +\theta(\epsilon).$ Also, there exists $\epsilon_1>0$, $\epsilon_1\leq \epsilon$, such that $F(p)\leq \lambda+\theta(\epsilon)/4$ $\forall$ $p\in \left(C(Q) \right)^{\epsilon_1}$. Since $\delta_n\rightarrow 0$ a.s., there exists $N=N(\omega)$ such that  such that $\forall n \geq N$, $F_n(p) < \lambda+ \theta(\epsilon)/3$  $\forall p \in (C(Q))^{\epsilon_1}$ and $F_n(p) > \lambda + \theta(\epsilon)/2$  $\forall p \in M_1$, so that $C(Q_n) \subset  (C(Q))^{\epsilon}$ , proving \eqref{eq-2.3}.  In order to show that $\delta_n\rightarrow 0$  a.s.  first note that, irrespective of $Q$, $|F(p) -  F(p')| \leq c\rho( p, p')$ where $c = 2\max\{\rho(q,q'): q,q'\in S\}$. Given any $\delta>0$, $|F(p) -  F(p')| < \delta/4 $ if $\rho(p,p') < \eta= \delta/4c.$ Let $q_1,\ldots,q_k$ be such that the balls $B(q_i:\eta)$ with radius $\eta$ and center $q_i$ cover $S$. Then $|F(p)  - F(q_i)| < \delta/4$ $\forall p\in  B(q_i:\eta)(i=1,\ldots,k).$ The same is true with $Q$ replaced by $Q_n$. By the strong law of large numbers (SLLN), there exists $N_1 = N_1(\omega; \delta)$ such that $|F_n(q_i) -F(q_i)| < \delta/2$ $\forall n \geq N_1$ ($i=1,\ldots k)$, outside a $P$-null set. It follows that, outside a $P$-null set,  $|F_n(p) -F(p|)|< |F_n(p) -F_n(q_i)| + |F_n(q_i)-F(q_i)| + |F(q_i) -F(p)| < \delta$ $\forall p\in B(q_i:\eta)$ $ (i=1,\ldots,k)$, provided $n \geq N_1.$

Consider now the non-compact case, but assuming \eqref{eq-compact}. Let $\lambda = \inf \{F(p): p\in S\}$. This infimum is attained in $S$. To see this, let $p_k$ ($k=1,2,\ldots$) be such that $F(p_k) \rightarrow \lambda$ as $k \rightarrow \infty$.  Since $\rho (p,q) \leq  \rho(p,x) + \rho(q,x)$ $\forall p,q,x$, one has
\begin{align}
\label{eq-2.4}
\rho(p,q) \leq  \int\rho (p,x)Q(dx)  + \int\rho(q,x)Q(dx) \leq  F^{1/2}(p) + F^{1/2}(q),  \forall p,q \in S.
\end{align}
 Letting $p= p_1$ and $q=p_k$, one obtains $\lim\sup_k  \rho(p_k,p_1) < \infty$. Hence the sequence $\{p_k\} $ is bounded, and its closure is compact,. Therefore, there exists $p^*$ such that $F(p^*) = \lambda$. Thus $C(Q)$ is nonempty and closed. If $q$ is any point in $C(Q)$ then taking $p=p^*$ and $q\in C(Q)$ in \eqref{eq-2.4}, one has $\rho(p^*,q) \leq 2\lambda^{1/2}$. That is $C(Q) \subset B(p^*,  \lambda^{1/2})$ . Thus part (a) is proved. To prove part (b), one has, using $Q_n$ for $Q$ and a fixed point $p^*$ for $q$ in $C(Q)$ in \eqref{eq-2.4}, the inequality  $F_n^{1/2}(p) \geq \rho(p,p^*)  -  F_n^{1/2}(p^*),$ $\forall p$.  Fix a $\delta >0$. Consider the compact set  $M_1 = \{q: \rho(q,p^*) \leq  2(\lambda+\delta)^{1/2} + \lambda^{1/2}\}$.  Then for $p \in S\backslash M_1$, one has $F_n(p) \geq [ 2(\lambda+\delta)^{1/2} + \lambda^{1/2}  - F_n^{1/2}(p^*)]^2 > \lambda + \delta$, $F_n(p^*) < \lambda+ \delta$ for all sufficiently large $n \geq N_1=N_1(\omega)$ except for $\omega$ lying in a $P$-null set, in view of the SLLN.  Hence $C(Q_n) \subset M_1$ for $n \geq N_1$.  Applying the result in the compact case (with $S= M_1$), one arrives at (b). Part (c) is an immediate consequence of part (b).

\end{proof}

For compact metric spaces $S$, part (c) of Theorem \ref{th-2.1} follows from \cite{ziezold}.

\begin{remark}
Theorem \ref{th-2.1} extends to more general Fr\'echet functions, including $F(p) = \int \rho^{\alpha}(p,q) Q(dq)$, $\alpha\geq 1$.
\end{remark}

\begin{remark}
 Relation \eqref{eq-2.3} does not imply that the sets $C(Q)$ and $C(Q_n) $ are asymptotically close in the Hausdorff distance. Indeed, in many examples $C(Q_n) $ may be a singleton, while $C(Q)$ is not. See, e.g., \cite{rabi03}, Remark 2.6, where it is shown that whatever be the absolutely continuous distribution $Q$ on $S^1$, $C(Q_n)$ is almost surely a singleton; in particular, this is the case when $Q$ is the uniform distribution for which $C(Q) = S^1$.  In view of this, and for asymptotic distribution theory considered later, it is important to find broad conditions on $Q$ for the \emph{existence of the Fr\'echet mean} (as the unique minimizer of the Fr\'echet function).
\end{remark}

Let $S=M$ be a \emph{differentiable manifold} of dimension $d$--a topological space which is  metrizable as a separable metric space such that (i) every $p \in  M$  has an open neighborhood up with a homeomorphism $\psi_p: U_p\rightarrow B_p$, where $B_p$ is an open subset of $\mathbb R^d$, and (ii) (compatibility condition)  if $U_p\cap U_q$ is nonempty, then the map $\psi_p\circ \psi_q^{-1}:  \psi_q(U_p\cap U_q)\rightarrow \psi_p(U_p\cap U_q)$ is a $C^{\infty}$  -a common example is the sphere $S^d = \{x\in \mathbb R^{d+1}: |x|=1\}$; one may take $p$ as the north pole (0,0,..,0,1) and $q$ as the south pole (0,0,\ldots ,0, -1),  $U_p= S^d\backslash \{q\}$ , $U_q= S^d\backslash\{p\},$ and $\psi_p$ and $\psi_q$ are the stereographic projections on $S^d\backslash\{q\}$ and $S^d\backslash\{p\}$ , respectively, onto $\R^d$.  Or, one may take 2$d$ open hemispheres $U_p$ of $S^d $ with poles whose coordinates are all zeros, except for +1 or - 1 at the $i$-th coordinate ($i=1,\ldots ,d$), each mapped diffeomorphically onto the open unit disc in $\mathbb R^d$. There are infinitely many distances  which metrize the topology of $M$. The two most common are (1) the Euclidean distance under an embedding, and (2) the geodesic distance when $M$ is endowed with a metric tensor.  For the first, recall that a smooth ($C^{\infty}$) map $J: M \rightarrow E^N$ is an embedding into an Euclidean space $E^N$, if (a) $J$ is one-to-one and $M\rightarrow J(M)$  is a homeomorphism with $J(M) $ given the relative topology of $E^N$, and (b) the differential $d_pJ $ on the tangent space $T_p(M)$ into the tangent space of $E^N$ at $J(p)$ is one-to-one.  The Euclidean distance on $J(M)$ (transferred to $M$ via $J^{-1}$) is called the \emph{extrinsic distance} $\rho_J$ on $M$. The embedding is said to be closed if $J(M)$ is closed. For $S^d$ one may, for example, take $J$ to be the inclusion map of $S^d$ into $\mathbb R^{d+1}$, and the extrinsic distance is the chord distance.

\begin{theorem}[\citep{rabi05} (Extrinsic Fr\'echet Mean on a Manifold)]
\label{th-2.2}
Let $M$ be a differentiable manifold and $Q$ a probability measure on it. If $J$ is a closed embedding of $M$ into an Euclidean space $E^N,$ and the Fr\'echet function of $Q$ is finite with respect to the induced Euclidean distance on $J(M)$), then the (extrinsic) Fr\'echet  mean exists as the unique minimizer of the Fr\'echet function if and only if there is a unique point $\mu_{J,E}$  in $J(M)$ closest to the Euclidean mean $m$ of the(push forward)  distribution $Q_J = Q\circ J^{-1}$  on $E^N$ , and then the extrinsic mean is $J^{-1} \mu_{J,E.}$
\end{theorem}

\begin{proof}   For a point $c\in J(M)$, writing $| y|^2 =\sum_{i=1}^N( y^{(i)})^2$  for the usual squared Euclidean norm on $E^N, $
\begin{align}
\int_{ J(M)}|c- y |^2 Q_J(dy) = \int_{E^N} |c -y|^2 Q_J(dy) =  \int_{E^N} | m-  y|^2  Q_J(dy) + |c - m |^2 .
\end{align}
This is minimized with respect to $c$, by setting $c$ to be the point in $J(M)$ closest to $m$ , if there is only one such  point, and the minimizer is not unique otherwise.
\end{proof}

\begin{example}[ Extrinsic Mean on the Sphere $S^d$]
Let  the inclusion map on $S^d$ into $\mathbb R^{d+1}$ be the embedding $J$. Then the mean $m$ of $Q_J$ on $\R^{d+1}$  lies inside the unit ball $B(0: 1)$ in $\R^{d+1}$ unless $Q$ is degenerate at a point $m\in S^d.$  If $Q$ is nondegenerate, the closest point to $m$ in $S^d$  is $m/|m|$ unless $m =0$ (i.e., $m$ lies at the center of the unit ball).  Thus (the image of ) the extrinsic mean is $\mu_{J,E} = m/|m|$. If $m=0$, then $C(Q)= S^d$.  If $Q$ is degenerate at $m$, then $m$ is the extrinsic mean. Taking $Q$ to be the empirical $Q_n$, the sample Fr\'echet mean is $\bar X/ |\bar X|,$ if $\bar X$ is not the origin in $\R^{d+1}$. If $\bar X  =0$, then $C(Q_n) = S^d.$
\end{example}
Theorem \ref{th-2.2} allows one in many important cases of interest in image analysis to find analytic characterizations for the existence of the extrinsic mean (i.e., as the unique minimizer of the Fr\'echet function) and computable formulas for its computation. This will be discussed in Section 3.

Unfortunately, on a Riemannian manifolds $M$ with metric tensor $g$ there is no good analog of Theorem \ref{th-2.2} for the \emph{intrinsic mean} of $Q$, -- the minimizer of the Fr\'echet function under the geodesic distance $\rho_g.$  The pioneering work by  \cite{karcher} followed by generalizations and strengthening, most notably, by \cite{kenws}, \cite{leh01} and \cite{Afsari2011} hold under support restrictions on $Q$, which are untenable for general statistical inference. The recent results of \cite{Afsari2011} are the sharpest among these, which we state below (for the Fr\'echet function \eqref{eq-frefunc}) without proof. For the terminology used in the statement we refer to the Appendix on Riemannian geometry. Recall that the support of a probability measure $Q$ on a metric space is the smallest closed set $D$ such that $Q(D)=1$.

\begin{theorem}[\citep{Afsari2011}  (Intrinsic Mean on a Riemannian Manifold)]
\label{th-2.3}  On a complete Riemannian manifold $(M,g)$, there exists an intrinsic Fre\'chet mean of $Q$, as the unique minimizer of the Fre\'chet  function \eqref{eq-frefunc} with the geodesic distance $\rho= \rho_g$, if the support of $Q$ is contained in a  geodesic ball of radius less than $r^* = (1/2)\min\{\text{inj} (M), \pi/\sqrt{\bar C}\}$. Here $inj(M)$ is the injectivity radius of $M$; and $\bar C$  is the supremum of sectional curvatures of $M$,  if positive, or zero otherwise.
\end{theorem}
\begin{remark}  If the Riemannian manifold $M$ is complete, simply connected and has non-positive curvature and the Fr\'echet function of $Q$ is finite, then the intrinsic mean of $Q$ exists (as the unique minimizer of $F$). An important generalization of this is to the so called \emph{metric spaces of non-positive curvature}, or the NPC spaces, which include many interesting metric spaces which are not manifolds. Such spaces were introduced by \cite{alex57} and further developed by \cite{ras68} and \cite{gromov1981structures}. See \cite{sturm2003probability} for a fine exposition.
\end{remark}

\begin{example}
\label{ex-2.2}
 Let $M = S^2.$ Then it has constant sectional curvature 1, and its injectivity radius is $\pi$. Thus if $Q$ has support contained in an open hemisphere, then the Fr\'echet mean of $Q$ under the geodesic distance exists. To see that one cannot relax this support condition in general, consider the uniform distribution on the equator. Then the minimum expected squared distance is attained at both the North and South poles (say, (0,0,1), and (0,0, -1)), so that $C(Q)$ has two points.
\end{example}

\begin{remark}  For purposes of statistical inference the support condition in Theorem \ref{th-2.3} is  restrictive, but as Example \ref{ex-2.2} shows one cannot dispense with the support condition without some further conditions on the nature of $Q$. In statistical practice a reasonable assumption is that the distribution is absolutely continuous.  In $S^1$ under the assumption that $Q$ has a continuous density (with respect to the arc length measure on intervals, i.e., the Lebesgue measure on [0, 2$\pi$) ) a necessary and sufficient condition, which applies broadly, was obtained in \cite{rabi07} and may be found  in \cite{rabibook}, pp. 31-33, 73-75.
\end{remark}

\section{Geometry of  Kendall's Shape Spaces. }




\subsection{Kendall's Similarity Shape Space $\Sigma_m^k$}

The similarity shape of a $k$-ad $x =(x_1,\cdots,x_k)$  in $\RR^m$, not all points the same, is its orbit under the group generated by translations, scaling and rotations. Writing  $\bar{x} =  (x_1+\cdots+x_k)/k$, $<\bar{x}> = (\bar{x},\cdots, \bar{x})$, the effect of translation is removed by looking at $(x_1- \bar{x},\cdots,x_k-\bar{x}) = x - <\bar{x}>$, which lies in the $mk - m$ dimensional hyperplane $L$ of $\RR^{mk}$ made up of $m\times k$ matrices with the $m$ row sums all equal to zero. To get rid of scale, one looks at $u= (x- <\bar{x}>)/ |x -<\bar{x}>|$, where $|. |$ is the usual norm in $\RR^{mk}$. This translated and scaled $k$-ad is called the \emph{preshape} of the $k$-ad.  It lies on the unit sphere in $L$, and is isomorphic to $S^{m(k-1) -1}$.  An alternative representation of the preshape, which we use, is obtained as $p= xH/|xH|$, where H is the $k\times(k-1)$ \emph{Helmert matrix} comprising $k-1$ column vectors forming an orthonormal basis of $1^{\perp}$, namely, the subspace of $\RR^k$ orthogonal to $(1,\cdots,1)'$.  A standard H has the $j$-th column given by $(a(j),\cdots,a(j), -ja(j), 0,\cdots, 0)'$, where the first $j$ elements equal $a(j) = [j(j+1)]^{-1/2}$ $(j=1,\cdots,k-1)$. Then $p$ is an $m\times(k-1)$ matrix of norm one. The \emph{shape} $\sigma(x) = \sigma(p)$ of $x$ is then identified with the orbit of $p$ under all rotations:
\begin{align}
\label{eq-6.1}
&\sigma(x)=\sigma(p)=\{Ap : A \in SO(m)\},
\end{align}
where $SO(m) = \{\text{A}: AA' =I_m, det(A) =1\} $  is called the \emph{special orthogonal group} acting on $\RR^m$.  The set of all shapes $\sigma(p)$ is Kendall's \emph{similarity shape} space $\Sigma_m^k.$

If $m=2$, $k>2$, the action of $SO(2)$ on the preshape sphere $S^{2k-3}$ is \emph{free}, i.e., no $A\in SO(2)$ other than the identity has a fixed point and each orbit of a point in $S^{2k-3}$ has an orbit of dimension one, namely the dimension of $SO(2)$. Since each $A\in SO(2)$ is an isometry of $S^{2k-3}$ endowed with the geodesic distance, it follows that $\Sigma_2^k=S^{2k-3}/SO(2)$ is a Riemannian manifold. For $m>2$, $k>m$, however, the action of $SO(m)$ on $S^{m(k-1)-1}$ is not free. For example, for $m=3$, each collinear $k$-ad in $S^{3(k-1)-1}$ is invariant under all rotations in $\mathbb R^3$ around the line of the $k$-ad. $\Sigma_3^k$ is then a disjoint union of two Riemannian manifolds, not complete, one comprising of the orbits of collinear $k$-ads under rotation by elements of $SO(2)$ other than those that keep it fixed (except for the identity). The other comprises of orbits under $SO(3)$ of all non-collinear $k$-ads in $S^{3(k-1)-1}$. $\Sigma_3^k$ is then a stratified space with two strata. More generally, $\Sigma_m^k$, $m>2$ ($k>m$), is a \emph{stratified space} with $m-1$ strata. See \cite{kBCL}, Chapter 6, for a complete description of the intrinsic geometry of $\Sigma_m^k$.  Also see \cite{Huck} for intrinsic analysis of more general stratified spaces of the form $M=N/\mathcal G$, where $N$ is a Riemannian manifold and $\mathcal G$ is a Lie group of isometries acting on $N$.

3.1(a).\textbf{Intrinsic geometry of $\Sigma_2^k$.}
For the case $m=2$, it is convenient to regard a $k$-ad $x =((x_1,y_1),\cdots,(x_k,y_k ))$ as a $k$-tuple $z= (z_1,\cdots,z_k )$ of numbers $z_1=x_1+iy_1,\cdots,z_k =x_k+iy_k$ in the complex plane $\mathbb{C}$, and let $p= (z- <\bar{z}>)/ |z -<\bar{z}>|$ . Then the shape of $p$, or $z$, is identified with the orbit
\begin{equation}
\label{eq-6.2}  \sigma(z) = \sigma(p) = \{ e^{i\theta}p: \theta\in (-\pi, \pi] \}.
\end{equation}
One may equivalently, consider the shape as the orbit $\{\lambda(z- <\bar{z}>) : \lambda\in\mathbb{C}\}$. That is, after Helmertization, the shape of $x$, or $z$, is identified with a complex line passing through the origin in $ \mathbb{C}^{ k-1}$. The shape space is then identified with the \emph{complex projective space}  $\mathbb{C}P^{k-2}$, of (real) dimension $2k-4$. We will, however, use the representation $\Sigma_2^k=\mathbb CS^{k-1}/G$,  where $G=\{ e^{i\theta}: \theta\in (-\pi, \pi] \}$ is a 1-dimensional compact group ($G \simeq S^1 )$ of isometries of the preshape sphere $\mathbb{C}S^{k-1} =\{p= (p_1,\cdots,p_{k-1})\in \mathbb{C}^{k-1}: |p|=1\}$, which is isomorphic to  $S^{2k-3}$.  Recall that the metric tensor on  $S^{2k-3} \simeq \mathbb{C}S^{k-1}$ is that inherited from the inclusion map into $\RR^{2(k-1)} = \{(x_1,y_1, x_2,y_2,  \cdots,x_{k-1},y_{k-1}): (x_j,y_j) \in \RR^2\; \forall j\}\simeq \mathbb{C}^{k-1} =\{(z_1, z_2,\cdots, z_{k-1}): z_j=x_j+iy_j \in \mathbb{C}\; \forall j \}$.  That is, the inner product at the tangent space $T_p\mathbb CS^k$ is $\langle \tilde v,\tilde w\rangle = Re(vw^*)$, when $\tilde v$,$\tilde w$ are expressed as complex $1\times(k-1)$ matrices (row vectors) in  $\mathbb{C} S^{k-1}$, satisfying $Re(p\tilde v^*)=0=Re(p\tilde w^*)$. The projection map is then $\pi: p\rightarrow \sigma(p)$.  The \emph{vertical subspace} $V_p$ is obtained by differentiating the curve $\theta\rightarrow e^{i\theta}p$, say at $\theta=0$, yielding $ip$. That is, $V_p = \{ cip:  c \in\RR\}$. Thus the \emph{horizontal subspace} is $H_p = \{\tilde{v}: Re(p\tilde{v}^*) =0, Re((ip)\tilde{v}^*) =0\} = \{\tilde{v}: p\tilde{v}^*=0\}$. The \emph{geodesics} $\gamma(t; \sigma(p),v)$ for $v = (d_p\pi)\tilde{v}$ (for $\tilde{v}$ in $H_p$), and the exponential map $Exp_{\sigma(p)}$ on $\Sigma_2^k$  are specified by this isometry between $T_{\sigma(p)}( \Sigma_2^k  )$ and $H_p$ for all shapes $\sigma(p)$ (See the Appendix, Section \ref{sec-app}). Thus, identifying vectors $v$ in $H_p$ with vectors $v$ in $T_{\sigma(p)}( \Sigma_2^k  )$, one obtains
\begin{align}
\label{eq-6.3}
&T_{\sigma(p)}( \Sigma_2^k  ) =\{ v= (d_p\pi)\tilde{v}: \forall v\;\text{ such that}\; p\tilde{v}^* =0\}\\ \nonumber
&Exp_{\sigma(p)} 0 =  \sigma(p), \;  Exp_{\sigma(p)} v = \sigma(\cos(|\tilde{v}|)p + \sin(|\tilde{v}|)\tilde{v}/|\tilde{v}| )\; ( v\neq 0,\; p\tilde{v}^* =0);\\ \nonumber
 &\gamma(t; \sigma(p),v) = \sigma ( (\cos t|\tilde{v}|)p + (\sin t|\tilde{v}|)\tilde{v}/|\tilde{v}| ), \;(t \in\RR, p\tilde{v}^* =0),\;v\neq 0.
 \end{align}
Denoting by $\rho_{gs}$ and $\rho_g$ the geodesic distances on $\mathbb{C} S^{k-1}$ and $\Sigma_2^k$, respectively, and recalling that (See  Example 10.1 and  \cite{kBCL}, p.114)  $\rho_{gs}(p,q) = \arccos(Re pq^*)$, one has
\begin{align}
\label{eq-6.4}
\rho_g (\sigma(p),\sigma(q)) &= \inf\{ \rho_{gs} (p,q): p \in O_u,  q \in O_w \}\\ \nonumber
& = \inf\{ \arccos (Re e^{i\theta}pq^*):\theta\in [0,2\pi)\}\\ \nonumber
& = \arccos(|pq^*|)  \in[0,\pi/2].
 \end{align}
It follows that the geodesics are periodic with period $\pi$, and the cut locus of $\sigma(p)$ is $\{ \sigma(q): \text{all}$ $q$  $\text{ such that} \arccos(|pq^*|) = \pi/2\}$, and that the injectivity radius of $\Sigma_2^k$ is $\pi/2$.  The inverse exponential map is given by $Exp_{\sigma(p)}^{-1}(\sigma(q)) = v$, where $ v= (d_p\pi)\tilde{v} \; (\tilde{v} \in H_p)$, and $\tilde{v}$ satisfies (Use \eqref{eq-50} with the representation of $S^{2k-3}$ as $\mathbb{C} S^{k-1}$)
\begin{align}
\label{eq-6.5}
 \tilde{v}& = Exp_p^{-1}(qe^{i\theta})\\ \nonumber
  &= [\arccos(Re(pq^*e^{-i\theta})](1-[Re(pq^*e^{-i\theta})]^2)^{-1/2}{ \left(qe^{i\theta} -( pq^*e^{-i\theta})p\right)},
  \end{align}
where $\theta$ is so chosen as to minimize  $\rho_{gs}(p, qe^{i\theta})= \arccos(Re(pq^*e^{-i\theta}))$. That is,
$(pq^*e^{-i\theta}) =|pq^*|$,   or $e^{i\theta}=pq^*/ |pq^*|$ ( for $pq^*\neq0$, i.e., for $\sigma(q)$ not in $C(\sigma(p))$.
 Hence, writing $\rho=  \rho_g(\sigma(p),\sigma(q))$, $\rho\neq0$, one has
\begin{align}
\label{eq-6.6}
 \tilde{v}&= [\arccos (|pq^*|)](1-|pq^*|^2)^{-1/2}\{ (pq^*/ |pq^*|) q -|pq^*|p\} \\ \nonumber                                                                                   
   &= [\rho/\sin \rho]  \{ qe^{i\theta} -(\cos \rho)p\}\;  \quad  \quad  \quad  \quad  (e^{i\theta}=pq^*/\cos \rho ).
\end{align}
This horizontal vector $\tilde{v}$ ($\in H_p$) represents $Exp_{\sigma(p)}^{-1}(\sigma(q)) = v$.

The sectional curvature of $\Sigma_2^k$  at a section generated by two orthonormal vector fields $\tilde{W}_1$ and $\tilde{W}_2$  is $1 + 3\cos^2\phi$ where $\cos\phi = \langle U_1, iU_2 \rangle$, $U_1$ and $U_2$ being the horizontal lifts of $\tilde{W}_1$ and $\tilde{W}_2$ (See \cite{docarmo}).

3.1(b).\textbf{ Extrinsic geometry of  $\Sigma_2^k$  induced by an equivariant embedding.} As mentioned in Section 2,  no broad sufficient condition is known for the  existence  of the intrinsic mean (i.e., of the uniqueness of the minimize of the corresponding Fr\'echet function). The extrinsic mean, on the other hand, is unique for most $Q$, and is generally computable analytically.  However, for an extrinsic analysis to be very effective one should choose a good embedding which retains as many geometrical features of the shape manifold as possible.  Let $\Gamma$ be a Lie group acting on a differentiable manifold $M$, and denote by $GL(N,\RR$) the \emph{general linear group }of nonsingular transformations on a Euclidean space $E^N$ of dimension $N$ onto itself. An embedding $J$ on $M$ into $E^N$ is said to be \emph{$\Gamma$-equivariant}  if there exists a group homomorphism $\Phi: \gamma\rightarrow \phi_{\gamma}$ of $\Gamma$ into $GL(N, \RR)$ such that  $J(\gamma p)) =  \phi_\gamma(Jp)$ $\forall p \in M$, $\gamma\in\Gamma$. Often, when there is a natural Riemannian structure on $M$, $\Gamma$ is a group of isometries of $M$.   Consider the so-called \emph{Veronese-Whitney embedding} $J$ of $\Sigma_2^k$  into the (real) vector space $S(k-1, \mathbb{C})$  of all $(k-1)\times(k-1)$ Hermitian matrices  $B=B^*$, defined by
\begin{equation}
\label{eq-6.9}  J\sigma(p) = p^*p \;  \;\quad \quad \quad \quad [\sigma(p) = \{e^{i\theta}p, \theta\in[0,2\pi),  p\in \mathbb{C} S^{k-1}\}  ].
\end{equation}
The Euclidean inner product on $S(k-1, \mathbb{C})$ , considered as a real vector space, is given by $\langle B,C\rangle = Re(\text{Trace} (BC^*))$.  Let $SU(k-1)$ denote the \emph{special unitary group} of all $(k-1)\times(k-1) $ unitary matrices $A$ (i.e., $A^*A = I$, $det (A) =1$) acting on $S(k-1, \mathbb{C}$)  by  $B\rightarrow  A^*BA$. Then the embedding \eqref{eq-6.9} is $\Gamma$- equivariant, with  $\Gamma = \{ \gamma_A: A \in SU(k-1)\}$ and the group action on $\Sigma_2^k$  given by:  $\gamma_A \sigma(p)= \sigma(pA)$. For  $J\sigma(pA)= A^*p^*pA =             \phi(\gamma_A ) (J\sigma(p))$, say, where the group homomorphism on $\Gamma$  onto $SU(k-1)$ is given by $\gamma_A \rightarrow \phi(\gamma_A ): \phi(\gamma_A )B= A^*BA.$ Note that $SU(k-1)$ is a group of isometries of $S(k-1, \mathbb{C})$. In the notation defining equavariance, one lets $S(k-1, \mathbb C)=E^N$ ($N=(k-1)^2$), $SU(k-1)$ is a subgroup of $GL(N, \mathbb R)$. 


To compute the extrinsic mean of $Q$ on $\Sigma_2^k$, let $Q_J=Q\circ J^{-1}$ be the probability induced on $S(k-1, \mathbb C)$ by the map $J$ in \eqref{eq-6.9}, and let $\mu_J$ denote its Euclidean mean. By Theorem \ref{th-2.1}, the (image of the) extrinsic mean of $Q$ is given by the orthogonal projection $P$ on $J(\Sigma_2^k)$.

\begin{proposition}\citep{rabi03}
\label{prop-3.1}
The image under $J$ of the extrinsic mean of $Q$ comprises all elements of the form $w^*w$ where $w^*$ is a normalized (column) eigenvector with the largest eigenvalues 
 of $\mu_J$. In particular, the extrinsic mean of $Q$ exists if and only if the largest eigenvalue of $\mu_J$ is simple.
\end{proposition}

\begin{proof}
Let $T$ be a $(k-1)\times (k-1)$ unitary matrix such that $T\mu_J T^*=D\equiv diag (\lambda_1,\lambda_2,\ldots, \lambda_{k-1})$ where $\lambda_1\leq \lambda_2\leq \cdots\leq \lambda_{k-1}$ are the ordered eigenvalues of $\mu_J$. Then the columns of $T^*$ form a complete orthonormal set of eigenvectors of $\mu_J$. By relabelling the landmarks, if necessary, we may assume that the $i$th column of $T^*$ is an eigenvector with eigenvalue $\lambda_i$. Write $\|A\|^2=\text{Trace}(AA^*)$ as the square of the Euclidean norm of $A$. Then for elements $w^*w$ of $J(\Sigma_2^k)=\{w^*w, w\in \mathbb CS^{k-1} \}$, denoting $v=wT^*$, one has 
\begin{align*}
\|w^*w-\mu_J\|^2&=\|Tw^*wT^*-T\mu_JT^*\|=\|v^*v-D\|^2\\
&=\sum_{i,j}|\bar v_iv_j-\lambda_i\delta_{ij}|^2=1+\sum_i\lambda_i^2-2\sum_i\lambda_i|v_i|^2
\end{align*}
which is minimized over $J(\Sigma_2^k)$ by taking $v_{k-1}=1$ and $v_i=0$ $\forall i <k-1$, i.e., by taking $w^*=T^*v^*$ be any normalized eigenvector of $\mu_J$ with the largest eigenvalue. 
\end{proof}

  A \emph{size-and-shape similarity shape } $s\sigma(z)$ is defined for Helmertized $k$-ads $z=(z_1,\cdots,z_{k-1})$ as its orbit under $SO(m)$. An equivariant embedding for it is $s\sigma(z)\rightarrow z^*z/|z|$, on the \emph{size-and-shape-similarity shape} space $S\Sigma_2^k$  into $S(k-1, \mathbb{C}).$

\subsection{Reflection Similarity Shape Space  $R\Sigma_m^k$, $m>2$, $k>m$.}

 For $m>2$,  let $\tilde{N}S^{m(k-1)-1}$ be the subset of the centered preshape sphere $S^{m(k-1)-1}$ whose points $p$ span $\RR^m$, i.e., which, as $m\times k$ matrices, are of full rank. We define the \emph{reflection similarity shape} of the k-ad as
 \begin{equation}\label{eq-6.10}
 r\sigma(p) = \{ Ap:  A \in O(m)\} \;  (p \in\tilde{N}S^{m(k-1)-1} ), 
\end{equation}
where $O(m)$ is the set of all $m\times m$ \emph{orthogonal matrices } $A: AA' = I_m$, $det (A) = \pm1$.  The set $\{r\sigma(p):  p \in\tilde{N}S^{m(k-1)-1} \}$ is the \emph{reflection similarity shape space} $  R\Sigma_m^k  = \tilde{N}S^{m(k-1)-1} /O(m)$.  Since $\tilde{N}S^{m(k-1)-1}$ is an open subset of the sphere  $S^{m(k-1)-1}$, it is a Riemannian manifold. Also $O(m)$ is a compact Lie group of isometries  acting on $S^{m(k-1)-1}$. Hence there is a unique Riemannian structure on $R\Sigma_m^k$  such that the projection map $p\rightarrow r\sigma(p)$ is a Riemannian submersion.

  We next consider a useful embedding of $R\Sigma_m^k$ into the vector space $S(k,\RR$) of all $ k\times k$ real symmetric matrices (See \cite{vic05}, \cite{Bandulasiri2009}, \cite{dryen2}, and \cite{abs2}). Define
\begin{equation}\label{eq-6.12}  J(r\sigma(p)) = p'p \;    (p \in \tilde{N}S^{m(k-1)-1}),
\end{equation}
with $p$ an $m\times(k-1)$ matrix with norm one. Note that the right side is a function of $r\sigma(p)$.  Here the elements $p$ of the preshape sphere are Helmertized.  To see that this is an embedding, we first show that $J$ is one- to-one on $R\Sigma_m^k$ into $S(k-1,\RR)$.  For this note that if $J(r\sigma(p))$ and  $J(r\sigma(q))$ are the same, then the Euclidean distance matrices $((|p_i-p_j|))_{1\leq i\leq j\leq k-1}$ and $((|q_i-q_j|))_{1\leq i\leq j\leq k-1}$ are equal.  Since $p$ and $q$ are centered, by geometry this implies that $q_i = Ap_i (i=1,\cdots,k-1)$ for some $A\in O(m)$, i.e., $r\sigma(p) =r\sigma(q)$.  We omit the proof that the differential $d_pJ $ is also one-to-one. It follows that the embedding is equivariant with respect to a group action isomorphic to $O(k-1)$.

\begin{proposition}[\citep{abs2}]
 \label{rc1}
(a) The projection of $\tilde\mu$ into $J(R\Sigma^k_m)$ is given by
\begin{equation}\label{r26}
P_{J(R\Sigma^k_m)}(\tilde\mu) = \{ A \colon A = \sum_{j=1}^m (\lambda_j - \bar \lambda + \frac{1}{m}) U_j {U_j}' \}
\end{equation}
where $\lambda_1 \ge \ldots \ge \lambda_k$ are the ordered eigenvalues of $\tilde\mu$, $U_1, \ldots, U_k$ are corresponding orthonormal  (column) eigenvectors and $\bar \lambda = \frac{\sum_{j=1}^m \lambda_j}{m}$.
(b) The projection set is a singleton and $Q$ has a unique extrinsic mean $\mu_E$ iff $\lambda_m > \lambda_{m+1}$. Then $\mu_E = \sigma(F)$ where $F = (F_1, \ldots, F_m)'$, $F_j = \sqrt{\lambda_j - \bar\lambda + \frac{1}{m}}U_j$.
\end{proposition}

For a detailed proof see \cite{abs2}, or \cite{rabibook}, pp. 114, 115.

        For $m>2$, a \emph{size-and-reflection shape } $sr\sigma(z)$ of a Helmertized $k$-ad $z$ in $\RR^m $ of full rank $m$ is given by its orbit under the group $O(m)$.  The space of all such shapes is the size-and-reflection shape space $SR\Sigma_m^k$ .  An $O(k-1)$-equivariant embedding of  $SR\Sigma_m^k$  into $S(k-1,\RR)$ is :  $J(sr\sigma(z)) = z'z/|z|$.

\subsection{  Affine Shape Space $A\Sigma_m^k$}

  Let $k> m+1$. Consider the set of all $k$-ads in $\RR^m$, with full rank $m$ as $m\times k$ matrices. The affine shape of a $k$-ad $x $ may be identified with its orbit under all affine transformations:
\begin{equation}\label{eq-6.13}
\sigma(x) = \{ Ax +c:  A\in GL(m,\RR), c\; \text{an} \; m\times k \;\text{matrix}
\}.
\end{equation}
If the $k$-ad is centered as $u= x - <\bar{x}>$, then the affine shape of $x$, or of $u$, is given by
\begin{equation}
\label{eq-6.14}
 \sigma(x) = \sigma(u) = \{Au: A\in GL(m,\RR)\},\;   (u\;\text{centered}\; k\text{-ad} \;\text{ of rank } \;m).
\end{equation}
The space of all such affine shapes is the \emph{affine shape space}  $A\Sigma_m^k$ .  Note that two Helmertized $k$-ads $u$ and $v$ (as $m\times(k-1)$ matrices of full rank) have the same shape if and only if the rows of $u$ and $v$ span the same m-dimensional subspace of $\RR^{k-1}$ . Hence  we can identify $A\Sigma_m^k$  with the Grasmannian $G_m(k-1)$, namely, the set of all
$m$-dimensional subspaces of $\RR^{k-1}$ \citep{sparr}.  For the Grassmann manifold, refer to \cite{boothby}, pp. 63, 168, 362, 363.  For extrinsic analysis on $A\Sigma_m^k\simeq G_m(k-1)$, consider the embedding of $A\Sigma_m^k$ into $S(k-1,\RR)$ given by
\begin{equation}\label{eq-6.15}
J(\sigma(u)) = FF',
\end{equation}
where $F =(f_1\cdots f_m)$  is a $(k-1)\times m$ matrix and $\{f_1,\cdots,f_m\}$ is an orthonormal basis of the $m$-dimensional subspace $L$, say, of $\RR^{k-1}$ spanned by the rows of $u$.  Note that the $(k-1)\times(k-1) $matrix $FF'$ is idempotent and is the matrix of orthogonal projection of $ \RR^{k-1}$ onto $L$. It is independent of the orthonormal basis chosen. The embedding is $O(k-1)$-equivariant under the group action $\sigma(u)\rightarrow \sigma(uO)$  $(O\in O(k-1))$ on $A\Sigma_m^k$, with $O(k-1)$ acting on $S(k,\RR)$ by $A\rightarrow OAO'$.

\begin{proposition}\citep{suga}\label{paf1}
The projection of $\tilde\mu$ into $J(A\Sigma^k_m)$ is given by
\begin{equation} \label{af2.1}
P(\tilde\mu) = \left\{ \sum_{j=1}^m U_j U_j' \right\}
\end{equation}
where $U = (U_1, \ldots, U_k) \in SO(k)$ is such that $\tilde\mu = U \Lambda U'$, $\Lambda = \mathrm{Diag}(\lambda_1,\ldots, \lambda_k)$,
$\lambda_1 \ge \ldots \ge \lambda_k = 0$.
The extrinsic mean $\mu_E$ exists if and only if $\lambda_m>\lambda_{m+1}$, and then $\mu_E = \sigma(F')$ where $F = ( U_1, \ldots, U_m )$.
\end{proposition}

For a proof see \cite{rabibook}, pp. 140, 141.

\subsection{ Projective Shape Space $P\Sigma_m^k$}

First, recall that the real projective space $\RR P^m$ is the space of all lines through the origin in $\RR^{m+1}$.  Its elements are  $[p]= \{\lambda p: \lambda \in \RR\backslash \{0\} \}$ for all $p \in  \RR^{m+1}\backslash \{ 0_{R^{m+1}}\}$. It is also conveniently represented as the quotient $S^m /G$ where $G$ is the two-point group $\{e, -e\}$, e being the identity map and $-ep = -p$ ($p \in S^m)$. That is, a line through $p$  is identified with $\{p/|p|, - p/|p|\}$ ( $p \in  \RR^{m+1}\backslash \{ 0_{R^{m+1}}\})$.  As a consequence, there is a unique Riemannian metric tensor on $\RR P^m  = S^m/G$ such that $p\rightarrow \{p,-p\}$ is a Riemannian submersion, with  $\langle u,v\rangle _{\RR P^m}  = u'v$ for all vectors $u$, $v$ in $T_{[p]}\RR P^m$.  The geodesic distance is given by $\rho_g([p],[q]) =  \arccos (|p'q|)\in [0,\pi/2]$, and the cut locus of $[p]$ is $Cut([p]) = \{[q]: \cos (|p'q|) =\pi/2\}$, so that the injectivity radius of $\RR P^m$ is $\pi/2$. Its sectional curvature is constant $+1$ (as it is of $S^m$).  The exponential map of $T_{[p]}\RR P^m$ (and its inverse on $\RR P^m \backslash (Cut([p]) )$ can be easily expressed in terms of those for the sphere $S^m$.  We will use [ ] for both representations.

The so-called \emph{Veronese-Whitney embedding} of $\RR P^m$ into $S(m+1, \RR)$ is  given by
\begin{equation}\label{eq-6.16}
 J([p]) = pp^t,  \;(p = (p_1,\cdots,p_{m+1})' \in S^m).
 \end{equation}
It is clearly $O(m+1)$-equivariant, with the group action on $\RR P^m$ as : $A[p] =[Ap]$ $(A \in O(m+1))$.

Turning to landmarks based projective shapes, assume $k> m+2$.  \emph{A frame of $\RR P^m$}  is a set of $ m+2$ ordered points $([p_1],\cdots,[p_{m+2}])$ such that every subset of $m+1$ of these points spans  $\RR P^m$  , i.e., every subset of $m+1$ points of $\{p_1,\cdots,p_{m+2}\}$ spans $\RR^{m+1}$. The \emph{standard frame} of  $\RR P^m$  is $([e_1], [e_2],\cdots,[e_{m+1}], [e_1+e_2+\cdots+e_{m+1}])$, where $e_i$  $(\in \RR^{m+1})$  has 1 in the ith position and zeros elsewhere.  A $k$-ad $y= (y_1,\cdots,y_k) = ([p_1],\cdots,[p_k])  \in( \RR P^m)^k$  is in \emph{general position} if there exist $i_1<i_2<\cdots<i_{m+2}$ such that  $(y_{i_1},\cdots,y_{i_{m+2}})$ is a  frame of  $\RR P^m$.   A \emph{projective transformation} $\alpha$ on  $\RR P^m$  is defined by
\begin{equation}\label{eq-6.17}
\alpha[p] = [Ap], \;  (p \in \RR^{m+1}\backslash\{0\})
\end{equation}
where $A \in GL(m+1,\RR)$.  The usual operation of matrix multiplication on $GL(m+1,\RR)$ then leads to a corresponding group of projective transformations on  $\RR P^m$.  This is the \emph{projective group} $PGL(m)$. Note that, for a given $A$ in $GL(m+1,\RR$),  $cA$ determines the same element of $PGL(m$) for all $c\neq 0$.  The \emph{projective shape} of a $k$-ad $y = (y_1,\cdots,y_k) = ([p_1],\cdots,[p_k])  \in (  \RR P^m )^k$  in general position is its orbit under $PGL(m)$:
\begin{align}
\label{eq-6.18}
 \sigma(y)& =\left\{ {\alpha y \equiv (\alpha[p_1], \cdots,\alpha[p_k]):
\alpha \in PGL(m)}\right\}, \; \\ \nonumber
(y &= ([p_1],\cdots,[p_k]\;\text{ in general position}).
\end{align}

The \emph{projective shape space} $PG\Sigma_m^k$ is the set of all projective shapes of $k$-ads in general position. Following \cite{mardiavic} and \cite{Patrangenaru2010}, we will consider a particular dense open subset of $PG\Sigma_m^k$. Fix a set of $m+2$ indices $I=\{i_j: j=1,\cdots,m+2\}$, $1\leq i_1<i_2<\cdots<i_{m+2}\leq k$. Define  $PG_I\Sigma_m^k$ as the set of shapes $\sigma(y)$  in $PG\Sigma_m^k$ , $y= (y_1,\cdots,y_k) = ([p_1],\cdots,[p_k]),$ such that every subset  of $m+1$ points of  $\{[p_{i_j}], j=1,\cdots,m+2\}$ spans $\RR P^m$.

The shape space $PG_I\Sigma_m^k$  (with $I=\{1,2,\cdots,m+2\}$) may be identified with $( \RR P^m)^{k -m-2}$ (See \cite{mardiavic}).

\section{ Asymptotic Distribution Theory for Fr\'echet Means.}
  
  This section is devoted to the asymptotic distribution theory of sample Fr\'echet means, which lies at the heart of statistical inference based on Fr\'echet means.  We first present a result which is broadly applicable to distributions on manifolds as well as more general locally Euclidean spaces such as \emph{stratified spaces}.  The basic idea behind it is rather simple. Suppose a probability $Q$ on a metric space $(S,\rho)$ has a Fr\'echet mean $\mu$.  Assume also the sample Fr\'echet mean $\mu_n$ converges to it (a.s. or in probability), which is true in particular under the topological assumption \eqref{eq-compact}. If, in local coordinates, $\mu$ and $\mu_n$ are expressed as $\nu$ and $\nu_n$ in an open subset of $\R^s$ for some $s$, then the Fr\'echet function $F_n$ of $Q_n$, expressed in local coordinates as $\tilde F_n$, say, satisfies a first order condition: $\text{grad}\; \tilde F_n(\nu_n) =0$. A taylor expansion of the left side around $\nu$, one expresses $\nu_n - \nu$ approximately as  $-\Delta^{-1}(\nu)\text{grad} \tilde F_n(\nu)$, where $\Delta$ is the Hessian of $\tilde F$  at $\nu$. Since grad $\tilde F_n(\nu)$ is the average of $n$  $s$-dimensional i.i.d. random vectors, the classical CLT is applied to show that $\sqrt{n}[\nu_n-\nu]$ is asymptotically normal. Here is the precise statement. For a detailed proof see \cite{linclt}, Theorem 3.3.  A slightly weaker version appears in \cite{clt13}.

Let $(S, \rho)$ be a metric space and $Q$  a probability measure on its Borel $\sigma$-field.  As before, define the \emph{Fr\'echet function} of $Q$ as
\begin{equation}
\label{eq-frechet}
	F(p) = \int \rho^2(p,q) Q(dq) \;  (p \in S).
\end{equation}
Assume that $F$ is finite on $S$ and has a unique minimizer $\mu = \argmin_p F(p)$. Then $\mu$ is called the \emph{Fr\'echet mean} of $Q$ (with respect to the distance $\rho$).  Under broad conditions, the Fr\'echet  sample mean $\mu_n$ of the empirical distribution $Q_n= \dfrac{1}{n}\sum_{j=1}^n \delta_{Y_j}$ based on independent $S$-valued random variables $Y_j$ ($j=1,\ldots, n$) with common distribution $Q$ is a consistent estimator of $\mu$. That is, $\mu_n\rightarrow \mu$ almost surely, as $n\rightarrow \infty$.  Here $\mu_n$ may be taken to be any measurable selection from the (random) set of minimizers of the Fr\'echet  function of $Q_n$, namely, $F_n(p) = \dfrac{1}{n}\sum_{j=1}^n  \rho^2(p,Y_j)$ (See Theorem \ref{th-2.1}).


The following assumptions are used in the proof of Theorem \ref{th-clt}.

\begin{itemize}
\item [(A1)] The Fr\'echet mean $\mu$ of $Q$ is unique.

\item [(A2)] $\mu\in G$,  where $G$ is a measurable subset of $S$, and there is a homeomorphism $\phi : G\rightarrow U$, where $U$ is  an open subset of $\mathbb R^s$ for some $s\geq 1$ and $G$  is given its relative topology on $S$.  The function 
\begin{equation}
\label{eq-asa2}
x\mapsto h(x;q) : = \rho^2(\phi^{-1}(x),q)\;
\end{equation}
is twice continuously differentiable on $U$, for every $q$  outside a $Q$-null set.
\item[(A3)]  $P(\mu_n\in G)\rightarrow 1$ as $n\rightarrow \infty$.

\item[(A4)]  Let $D_rh(x;q) = \partial h(x;q)/\partial x_r, r=1,\ldots,s$. Then
\begin{equation}
	E | D_rh (\phi(\mu);Y_1)|^2<\infty, \; E |D_{r,r^\prime} h(\phi(\mu);Y_1)|<\infty \;\text{for}\; r,r^\prime =1,\ldots, s.
\end{equation}
\item [(A5)]  Let $u_{r,r^\prime}(\epsilon;q)= \sup\{| D_{r,r^\prime} h(\theta;q) - D_{r,r^\prime} h(\phi (\mu);q)| : |\theta -\phi(\mu)|<\epsilon\}$.  Then
\begin{equation}
	E |u_{r,r^\prime}(\epsilon;Y_1)|\rightarrow 0 \;\text{as}\; \epsilon\rightarrow 0\;\text{for all}\; 1\leq r, r^\prime\leq s.
\end{equation}
\item [(A6)]  The matrix   $\Lambda= [E D_{r,r^\prime} h(\phi(\mu);Y_1)]_{r,r^\prime=1,\ldots,s }$ is nonsingular.
\end{itemize}

\begin{remark}
\label{rem-2.1}
Observe that $Eh(x, Y_1)=F(\phi^{-1}(x))=ED_rh(x,Y_1)=D_rF(\phi^{-1}(x))$, $1\leq r\leq s$, $x\in U$. Also, $ED_rh(\phi (\mu), Y_1)=D_rF(\phi^{-1}(x))\mid_{x=\phi(\mu)}=0$, $1\leq r\leq s$, since $F(\phi^{-1}(x))$ attains a minimum at $x=\phi(\mu)$.
\end{remark}

\begin{theorem}[\citep{linclt}]
\label{th-clt}
Under assumptions (A1)-(A6) ,f
\begin{equation}
\label{eq-maineq}
	n^{1/2}[ \phi(\mu_n)-\phi(\mu)] \xrightarrow{\mathcal{L}}  N(0, \Lambda^{-1}C \Lambda^{-1}), \;\text{as}\;  n\rightarrow \infty,
\end{equation}
where  $C $ is the covariance matrix of $\{D_rh(\phi(\mu);Y_1), r=1,\ldots,s\}$.
\end{theorem}

\begin{proof}
 The function $x\rightarrow F_n(\phi^{-1}x) = \dfrac{1}{n}\sum_{j=1}^n h(x,Y_j)$ on $U$ attains a minimum at  $\phi(\mu_n) \in U$ for all sufficiently large $n$ (almost surely).  For all such $n$ one therefore has the first order condition
\begin{equation}
\label{eq-firstorder}
\nabla \; F_n(\phi^{-1} \nu_n) = \dfrac{1}{n} \sum_{j=1}^n \nabla \; h(\nu_n,Y_j)  =0,
\end{equation}
where $\nu= \phi( \mu)$, $\nu_n= \phi( \mu_n)$ (column vectors in $U$). Here $\nabla$ is the gradient $(D_1,\ldots, D_r).$  A Taylor expansion yields
\begin{equation}
\label{eq-07}
	0 =  \dfrac{1}{n}\sum_{j=1}^n \nabla \;h(\nu_n,Y_j)  = \dfrac{1}{n}\sum_{j=1}^n \nabla \;h(\nu,Y_j)  +  \Lambda_n (\nu_n-\nu)
\end{equation}
where $\Lambda_n$  is the $s\times s$ matrix given by
\begin{equation}
\label{eq-08}
	\Lambda_n = \dfrac{1}{n}\sum_{j=1}^n [D_{r,r^\prime}h(\theta_{n,r,r^\prime},Y_j)]_{r,r^\prime=1,\ldots,s},
\end{equation}
and $\theta_{n, r, r^\prime}$ lies on the line segment joining $\nu_n$ and $\nu$.  We will show that
\begin{equation}
\label{eq-09}
	\Lambda_n \rightarrow \Lambda \;\text{in probability},\; \text{ as } n\rightarrow \infty.
\end{equation}
Fix  $r,r^\prime\in \{1,\ldots,s\}$.  For $\delta>0$, write $E u_{r,r^\prime} (\delta,Y_1) = \gamma(\delta)$. There exists $ n=n(\delta)$ such that  $ P(|\nu_n-\nu|>\delta)<\delta$ for $n>n(\delta).$  Now
\begin{align}
	E\big| [ \dfrac{1}{n}\sum_{j=1}^n D_{r,r^\prime}h(\nu_n,Y_j)- \dfrac{1}{n}\sum_{j=1}^n D_{r,r^\prime}h(\nu,Y_j)]\cdot 1_{[|\nu_n-\nu|\leq \delta]}\big|&\nonumber\leq E  \dfrac{1}{n}\sum_{j=1}^n u_{r,r^\prime} (\delta,Y_j)\\ \nonumber
 &= E u_{r,r^\prime} (\delta,Y_1) = \gamma(\delta) \rightarrow 0
\end{align}
 as  $\delta\rightarrow 0$.
Hence, by Chebyshev's inequality for first moments, for $n> n(\delta) $ one has for every $\epsilon>0$,
\begin{equation}
\label{eq-11}
	P(\big |  \dfrac{1}{n}\sum_{j=1}^n D_{r,r^\prime}h(\nu_n,Y_j)- \dfrac{1}{n}\sum_{j=1}^n D_{r,r^\prime}h(\nu,Y_j)\big|> \epsilon) \leq  \delta + \gamma(\delta)/\epsilon \rightarrow 0\;\text{ as}\;  \delta\rightarrow 0.
\end{equation}
This shows that
\begin{equation}
\label{eq-12}
	\big[ \dfrac{1}{n}\sum_{j=1}^n D_{r,r^\prime} h(\nu_n,Y_j)-  \dfrac{1}{n}\sum_{j=1}^n D_{r,r^\prime}h(\nu,Y_j)\big] \rightarrow 0;\text{ in probability as}\; n\rightarrow \infty.
\end{equation}
Next, by the strong law of large numbers,
\begin{equation}
\label{eq-13}
	 \dfrac{1}{n}\sum_{j=1}^n  D_{r,r^\prime}h(\nu,Y_j) \rightarrow  ED_{r,r^\prime}h(\nu,Y_1)\;\text{ almost surely,}\;\text{ as}\; n\rightarrow \infty.
\end{equation}
Since \eqref{eq-11} -- \eqref{eq-13} hold for all $r$,$r^\prime$, \eqref{eq-09} follows.  The set of symmetric $s\times s$ positive definite matrices is open in the set of all $s\times s$ symmetric matrices, so that \eqref{eq-09} implies that $\Lambda_n$ is nonsingular with probability going to 1 and  $\Lambda_n^{-1}\rightarrow \Lambda^{-1}$ in probability, as $n\rightarrow \infty$.  Note that $E\nabla h(\nu, Y_1)=0$ (see Remark \ref{rem-2.1}). Therefore, using (A4),  by the classical CLT and Slutsky's  Lemma, \eqref{eq-07} leads to
\begin{equation}
\label{eq-14}	\sqrt{n}(\nu_n-\nu) = \Lambda_n^{-1}[-(1/\sqrt{n})  \dfrac{1}{n}\sum_{j=1}^n \nabla\; h(\nu,Y_j)]  \xrightarrow{\mathcal{L}} N(0, \Lambda^{-1}C\Lambda^{-1}),
\end{equation}
as $ n\rightarrow \infty$.
\end{proof}

  For the case of the \emph{extrinsic mean}, let $M$ be a $d$-dimensional differentiable manifold, and $J: M\rightarrow  E^N$ an embedding of $M$ into an $N$-dimensional Euclidean space. Assume that $J(M)$ is closed in $E^N$, which is always the case, in particular,  if $M$ is compact. The extrinsic distance $\rho_{E,J}$ on $M $ is defined as $\rho_{E,J}(p,q) = |J(p)-J(q)|$ for $p,q \in M$, where  $|\cdot|$ denotes the Euclidean norm of $E^N$. The image $\mu$ in $J(M)$ of the \emph{extrinsic mean} $\mu_{E,J }$ is then given by $\mu= P(m)$, where $m$ is the usual mean of $ Q\circ J^{-1}$ thought of as a probability on the Euclidean space $E^N$, and $P$ is the orthogonal projection defined on an $N$-dimensional neighborhood $V$ of $m$ into $J(M)$ minimizing the Euclidean distance between $p\in V$ and $J(M)$.  If the projection $P$ is unique on $V $ then the projection $\mu_n= P(m_n)$ of the Euclidean mean $m_n = \sum_{j=1}^n J(Y_j)/n$ on $J(M)$ is, with probability tending to one as $n\rightarrow \infty$, unique and lies in an open neighborhood $G$ of $\mu =P(m)$ in $J(M)$.  Theorem \ref{th-clt} immediately implies the following result of \cite{rabi03} (Also see \cite{rabibook}, Proposition 4.3).
  Assume that $P$ is uniquely defined in a neighborhood of the $N$-dimensional Euclidean mean $m$ of $Q\circ J^{-1}$.  Let $\phi$ be a diffeomorphism on a neighborhood $G$  of $\mu= P(m)$ in $J(M)$ onto an open set $ U$ in $\mathbb R^d$. Then, using the notation of \eqref{eq-maineq},
  \begin{equation*}
  \sqrt{n} \left[\phi(\mu_n)-\phi(\mu) \right]= \sqrt{n}\left[\phi(P(m_n))- \phi( P(m))\right] \xrightarrow{\mathcal{L}} N(0, \Lambda^{-1}C\Lambda^{-1}), \;\text{as}\; n\rightarrow \infty.
  \end{equation*}
  


One may, in particular, choose $(U,\phi)$ to be a coordinate neighborhood of $\mu=P(m)$ in $J(M)$. In \cite{rabi03}, however, $\phi$ is chosen to be the linear orthogonal projection on $G$ into the tangent space $T_{\mu} J(M)$.

For a more computable expression  of the limit, let $X_j$, $1\leq j\leq n$, be i.i.d, $M$-valued observations with common distribution $Q$, and $Y_j=J(X_j)$, $1\leq j\leq n$. In a neighborhood $V$ of $m$, the differential $d_yP$ maps $T_jE^N\approx E^N$. One expresses $d_{P(m)}e_j=\sum_{i=1}^db_{ij}F_i$, 
\begin{align*}
d_m(\bar Y-m)&=\sum_{j=1}^N\sum_{i=1}^d b_{ji}(\bar Y-m)(j)F_i\\
&=\sum_{i=1}^d\left(\sum_{j=1}^N b_{ji}(\bar Y-m)^{(j)}  \right)F_i\;\;\quad \quad (F_i=F_i(P(m))).
\end{align*}
Thus one arrives at the following result.

\begin{proposition}
\label{prop-4.3}
Assume the projection $P$ is uniquely defined and is continuously differentiable in a neighborhood $V$ of $m=EY_j$, and $E|Y_j|^2<\infty$. Then 
\begin{align}
\label{eq-exclt}
\sqrt{n} d_mP(\bar Y-m)\xrightarrow{d} N(0,\Sigma)
\end{align}
where $\Sigma=B'CB$ with $b=((b_{ji}))$ and $C$ is the $N\times N$ covariance matrix of $Y_j$. 
\end{proposition}


\begin{corollary}[CLT for Intrinsic Means-I]
 \label{coro-1}  Let $(M,g)$ be a $d$-dimensional complete Riemannian manifold with metric tensor $g$ and geodesic distance $\rho_g$. Suppose  $Q$ is a probability measure on $M $ with intrinsic mean $\mu_I$, and that $Q$ assigns zero mass to a neighborhood, however small, of the \emph{cut locus} of $\mu_I$.  Let $\phi= Exp \mu_I^{-1}$ be the inverse exponential, or $\log$-, function at $\mu_I$ defined on a neighborhood $G$ of $\mu = \mu_I$ onto its image $U$ in the tangent space $T_{\mu_I} (M)$. Assume that the assumptions (A4)-(A6) hold. Then, with $s=d$, the CLT  \eqref{eq-maineq} holds for the intrinsic sample mean $\mu_n = \mu_{n,I}$, say.
\end{corollary}

\begin{remark} In addition to providing a CLT for manifolds (of dimension $d$), Theorem \ref{th-clt} applies to many stratified spaces which are manifolds of different dimensions s glued together. See \cite{linclt} for a simple derivation of a CLT for the so called Open Book model, originally due to \cite{openbook}. Another stratified space to which Theorem \ref{th-clt} applies is $\Sigma_m^k$, $m>2$, $k>m$, described in Section 3 (see Remark \ref{rem-4.17}).
\end{remark}

\begin{remark}
  For manifolds of dimension d (i.e., $s =d$),  Theorem \ref{th-clt} and Corollary \ref{coro-1} improve upon Theorem 2.3 and 5.3 in \cite{rabibook} (and  earlier results in \cite{rabi05}).
  \end{remark}

We now turn to the derivation of the asymptotic distribution of sample intrinsic Fr\'echet means on Riemannian manifolds which does not require the support restriction of Corollary \ref{coro-1}.

For the case of the circle $S^1$, necessary and sufficient conditions for the existence of the intrinsic mean was established in \cite{rabi07}, under the assumption of a continuous density with respect to the uniform distribution.  The result also along with a central limit theorem for the sample intrinsic mean in \cite{rabibook}, pp. 31-34, 72-75. A proof of the CLT was also obtained independently in \cite{mickilliam}.  Some additional results, especially for distributions with discontinuous density may be found in the recent article \cite{Hotz2015}.

\begin{proposition} 
\label{prop-4.4}Suppose a complete orientable $d$-dimensional Riemannian manifold $(M,g)$ has the property that the image $D \subset T_pM$ of $M\backslash\text{Cut}(p)$ under the map $\log_p = Exp_p^{-1}$ is the same for all $p\in M$ , and the push forward of the volume measure on $D$ under the $\log_p$ map is also  the same for all $p$.  Assume that the intrinsic mean $\mu_i$ of a probability $Q$ on $M$ exists, and that $Q$ is absolutely continuous in a neighborhood $W$ of $\text{Cut}(\mu_I)$ with a density $f$ on $W$ which is twice continuously differentiable.  Assume also that the first and second derivatives of $p\rightarrow f(Exp_pv)$, in local coordinates,  are bounded for $p$ in a neighborhood of $\mu_I$ by functions $f_i(v$),  $f_{ij}(v)$ such that, for a sufficiently small $\epsilon>0$,
 \begin{align}
  \int_{\{R-\epsilon<|v|<R\}}|v|^2f_i(v)m(dv)  < \infty ,  \int_{\{R-\epsilon<|v|<R\}} |v|^2f_{ij}(v)m(dv)  < \infty,\;\; (i,j=1,\ldots,d),                    
   \end{align}
where $m(dv)$ is the push forward on $T_pM$ of the volume measure by the map $\log_{\mu_I}$.  Then there exists a neighborhood of $\mu_I$ in which the Fr\'echet function \eqref{eq-frefunc} with  $\rho=\rho_g$, is twice continuously differentiable.
\end{proposition}

\begin{proof}
First note that there exist $r>0$ and $\epsilon>0$, both sufficiently small and a geodesic ball $B_r$ with center $\mu_I$ and radius $r>0$ such that ($\text{Cut} (B_r))^{\epsilon} \subset W$, where $\text{Cut} (B_r) = \cup \{\text{Cut}(p): p\in B_r\}$ and $A^\epsilon$  is the $\epsilon$-neighborhood of a set $A \subset M.$ For $p \in B_r$,
\begin{align} 
\label{eq-int}
F(p) = \int_{\{q: |\log_pq| <R-\epsilon\}} \rho_g^2(p, q)Q(dq) +   \int_{\{q: R-\epsilon<|\log_pq| <R\}} \rho_g^2(p, q)Q(dq)                                                     
   \end{align}
The first integral in \eqref{eq-int} is clearly twice continuously differentiable. The second integral may be expressed as
\begin{align}
\int_{\{R-\epsilon < |v| <R\}} |v|^2f(Exp_pv)m(dv),
\end{align}
where $R = \rho_g(p, Cut(p))$.

\end{proof}

\begin{remark}
  We conjecture that the conclusion of Proposition \ref{prop-4.4} holds much more generally and, in particular, for all compact orientable Riemannian manifolds, if $Q$ has a twice continuously differentiable density. 
  \end{remark}

The following result is an immediate consequence of Proposition \ref{prop-4.4}.


\begin{corollary}
  \label{coro-4.5}
    Let $M=S^d$ with the usual Riemannian metric tensor, and $Q$ a probability measure on it. 
 (a) Then the Fr\'echet function is twice continuously differentiable if $Q$ has a twice continuously differentiable density. (b) If $Q $ is absolutely continuous in a neighborhood of the cut locus $Cut(p)$ of a point $p$,  with a twice continuously differentiable density there, then the Fr\'echet function is twice continuously differentiable in a neighborhood of  $p$.
 \end{corollary}
 
For the statement of the next result we continue to use the notation of Corollary \ref{coro-1}.

 \begin{theorem}[\citep{linclt}]
\label{th-2.4}
Suppose that the intrinsic mean  $\mu$ of $Q$ exists, and that $Q$ is absolutely continuous in a neighborhood $W$ of the cut locus of $\mu$ with a continuous density with respect to the volume measure. Assume also that (i) $Q(Cut(B(\mu;\epsilon))) = O(\epsilon^{d-c})$, $\epsilon\rightarrow 0$,  for some $c$, $0\leq c< d$, (ii) on some neighborhood $V$ of $\nu = \phi(\mu) =0$ the function $\theta\rightarrow F\left(\phi^{-1}(\theta)\right)$  is twice continuously differentiable with a nonsingular Hessian $\Lambda(\theta)$, and (iii) (A4) holds with $\phi(\mu)$ replaced by $\theta $, $\forall$$\theta\in V$. Then, if $d>c+2$, one has the CLT \eqref{eq-maineq}  for the sample intrinsic mean $\mu_n$.
\end{theorem}

\begin{proof}
One may take the neighborhood $V $ of $\nu=0$ sufficiently small such that $Cut(\phi^{-1}(V))\subset W$. Then $Z_n(\theta):= n^{-1}\sum_{1\leq j\leq n}\text{grad}\;h(\theta,Y_j)$ is well defined for $Y_j\not\in Cut(\phi^{-1}\theta)$, $j=1,\ldots, n$, that is,   with  probability one,  provided $\theta\in V$, since  $Q(Cut(\phi^{-1}\theta))=0$. By the classical CLT, $Z_n(0):= n^{-1}\sum_{1\leq j\leq n}\text{grad}\;h(0,Y_j) $ is of the order $O_p(n^{-1/2}).$  Let $B_n$  be the ball in $T_{\mu}M$ with center $\nu= \phi(\mu)=0$ and radius $n^{-1/2}\log n$. By hypothesis, the probability that  $Y_j\in Cut(\phi^{-1}(B_n))$ is $O((n^{-1/2}\log n)^{d-c} )$.  For $\phi^{-1}(B_n)$ is the geodesic ball $B(\mu; n^{-1/2}\log n)$, hence the probability that the set $ \{Y_j:j=1,\ldots, n\}$  intersects $Cut(\phi^{-1}(B_n))$  is $O(n(n^{-1/2}\log n)^{d-c} ) = o(1)$ if $d>c+2$. Therefore, with probability converging to 1, one may use a Taylor expansion of $Z_n(\theta)$ in $B_n$, 
 \begin{align}
Z_n(\theta) = Z_n(\nu) + \Lambda_n(\theta)(\theta-\nu),\;\; (\theta\in B_n), \; (\nu=0),     
\end{align}                                                                       
where $\Lambda_n(\theta)$ is the $d\times d$ matrix whose $(r,r') $ element is  $n^{-1}\sum_{1\leq j\leq n}D_{r,r'}h(\theta(n;r,r', Y_j), Y_j)$  with $\theta(n;r,r',Y_j)$ lying on the line segment joining $\theta$ and $\nu=0$. 
By hypothesis (ii), with probability converging to one as $n\rightarrow \infty$, $\Lambda_n(\theta)$  is nonsingular for all large $n$ ($\theta\in B_n$) since its difference (in norm) from the Hessian $\Lambda(\theta)$  goes to zero as $n\rightarrow \infty$, by the strong law of large numbers. Also, with probability going to 1, the function $\theta\rightarrow H_n(\theta) =0 -\Lambda_n(\theta)^{-1}Z_n(\nu)$ maps  $\bar{B}_n$   into itself, where  $\bar{B}_n$  is the closure of $B_n$. For this argument recall that $Z_n(0) = O_p(n^{-1/2})$ by the classical CLT. By the \emph{Brouwer fixed point theorem} \citep{milnor1965topology}, $H_n(\theta) $ has a fixed point. Let $\nu_n$ denote a measurable selection from the set of fixed points in  $\bar{B}_n$. It follows that, with probability going to 1, $\nu_n$ converges to $\nu$ and satisfies the first order equation \eqref{eq-firstorder},   and $\nu_n$ is the sample intrinsic mean, since the Fr\'echet function is strictly convex in a neighborhood of $\nu$. The CLT now follows as in the last sentence and relation \eqref{eq-14} of the proof of Theorem  \ref{th-clt}.  

\end{proof}

%

\begin{corollary} 
\label{coro-2.5}
Suppose $Q$ on $M= S^d$ ($d>2$) has an intrinsic mean $\mu$ and is absolutely continuous on a neighborhood $W$ of $Cut(\mu)$ with a continuous density on $W$.  Then the CLT for the sample intrinsic mean holds.
\end{corollary}

\begin{remark}
 \label{rem-4.4}
 It follows from the proof of Theorem 2.2, specialized to $S^d$ , that the Hessian $\Lambda$ is positive definite.
 \end{remark}

 \begin{remark}

Although it is curious that the proof of Theorem \ref{th-2.4}  does not hold for $d=2$, the authors expect that a proof of Corollary \ref{coro-2.5}  for the case $d=2$ may be given using polar coordinates.  For the moment,  the CLT for $S^2$  is derived only under the support  restriction of Corollary \ref{coro-1}. 
\end{remark}

\begin{remark}
\label{remark-3.8}
 Suppose $\mathcal G$ is a Lie group of isometries on $S^d$, $d>2$. Then the projection $\pi: S^d\rightarrow S^d/\mathcal G$  is a \emph{Riemannian submersion} on $S^d$ onto its quotient space $M= S^d/\mathcal G$ (\cite{gl} , pp. 63-65, 97-99). Let $Q$ be a probability measure on $S^d$ with a twice continuously differentiable density and a Karcher or  intrinsic mean $\mu$. Let $\tilde \mu$ be the projection of $\mu$ and $\tilde \mu$ that of the sample intrinsic (or Karcher) mean $\mu_n$. Then, in local coordinates, the differential of the Fr\'echet function on $M$ vanishes at $\tilde \mu$, because $\pi$ is smooth and the differential of the Fr\'echet function on $S^d$ vanishes at $\mu$. The delta method provides a CLT for the corresponding sample Fr\'echet mean $\tilde \mu_n$  in local coordinates. However, $\tilde \mu$ ($\tilde \mu_n$) are unlikely to be the intrinsic (respectively, sample intrinsic) mean of $\tilde Q$ (respectively,  $\tilde Q_n$) obtained from $Q$ (respectively, $Q_n$) by the projection map.
 \end{remark}

 \begin{remark}
    One may also explore the opposite route for a probability $\tilde Q$ on $M$ with a density and a unique intrinsic/Karcher mean $\tilde \mu$  and a probability $Q$, among a fairly large family of distributions with smooth densities on $S^d$  whose projection on $M$  is $\tilde Q$, such that $Q$ satisfies the hypothesis of Corollary \ref{coro-2.5} with  $\pi(\mu) =\tilde \mu$ . One may then apply the CLT on $S^d$  to derive one on $S^d/\mathcal G$. As an example consider the antipodal map $g(p)=-p$, and $\mathcal G= \{g, \text{identity}\}$. Let $\tilde Q$ be a probability on $M= S^d/\mathcal G = \mathbb RP^d$ (the real projective space) thought of as a probability on the upper hemisphere vanishing smoothly at the boundary, and with a unique intrinsic mean $\tilde \mu=\{\mu,-\mu\}$, where $\mu$ is the Karcher mean of $Q $ (restricted to the hemisphere). This opens a way for CLT's on Kendall's shape spaces as well.

\end{remark}

\begin{remark}
As indicated in Remark \ref{remark-3.8}, one of the significances of a CLT on $S^d$ is that it may provide a  route to intrinsic CLTs on $S^d/\mathcal G$, the space of orbits under a \emph{Lie group $\mathcal G$ of isometries} of $S^d$.  Such spaces include the so-called \emph{axial spaces} (or \emph{real projective spaces} $\mathbb R P^d$), and Kendall type  shape spaces which are important in \emph{shape-based image analysis}.  For the latter spaces, $S^d$ is the so-called \emph{preshape sphere}. Observe  that  the hypothesis (i) of Theorem \ref{th-2.4} may not hold in all such quotient spaces. For example, on $\mathbb R P^d$ one only has the order $O(\epsilon)$ in hypothesis (i) in Theorem \ref{th-2.4}, since the cut locus of the a point in $\mathbb{R} P^d$ is isomorphic to $\mathbb{R} P^{d-1}$.   For Kendall's planar shape space, identified as the \emph{complex projective space} $\mathbb{C}P^{k-2}$, of dimension $d=2k-4$, the volume measure of $Cut\left(B(\mu;\epsilon)\right)$ is $O(\epsilon^{2})$,  since the cut locus of a point of $\mathbb{C}P^{k-2}$ is isomorphic to $\mathbb{C}P^{k-3}$. For these facts one may refer to \cite{gl}, Section 2.114, pp. 102, 103. 
\end{remark}

\begin{remark}
\label{rem-4.17}
For $m>2$, $k>m$, $\Sigma_m^k$ is a stratified space in the intrinsic topology. But the projection $S^{m(k-1)-1}\rightarrow \Sigma_m^k$ (see \eqref{eq-6.1}) is continuous and $\Sigma_m^k$ is a compact metric space.  Hence Theorem \ref{th-clt} still applies to this stratified space. 

\end{remark}

\section{ Nonparametric Inference on General Manifolds  }

Theorems \ref{th-2.1}, \ref{th-2.4} allow us to construct nonparametric confidence regions for intrinsic and extrinsic means of probability measures $Q$ on a manifold $M$, and to carry out nonparametric two-sample tests for the equality of such means of two distributions $Q_1$ and $Q_2$ on $M$.  The latter tests are really meant to distinguish $Q_1$ from $Q_2$. On high dimensional spaces, such as the shape spaces of main interest here, the means are generally good indices for this purpose, as the data examples in Section 10 show.

  For the construction of an extrinsic confidence region for the extrinsic mean $\mu_E$ of $Q $ one may use the corresponding region for $\mu_J$ using \eqref{eq-exclt}   and then transform by $J^{-1}$.  
  The following asymptotic chisquare distribution is an easy consequence of Proposition \ref{prop-4.3}:
\begin{equation}
\label{eq-5.1}
n\left[(d_{\bar{Y}}P) (\bar{Y}-m)\right]'( \hat{B}'\hat{\Sigma}\hat{B})^{-1} \left[(d_{\bar{Y}} P)( \bar{Y}- m) \right]\rightarrow \chi^2_d\;\;\text{in distribution,}
\end{equation}
where  $\chi^2_d$ is the chisquare distribution with $d$ degrees of freedom. Here $\hat{B} = B( \bar{Y})$ estimates $B= B(m)$, and $\hat{\Sigma}$ is the sample covariance matrix of $Y_1,\cdots,Y_n$. The statistic does not depend on the choice of the orthonormal basis of $T_{\bar{Y}} (J(M))$ for computing $\hat{B}$. The relation \eqref{eq-5.1} may be used to construct extrinsic mean $\mu_E=J^{-1}P(m).$  Bootstrapping, which leads to a smaller order of coverage error in the case of an absolutely continuous $Q,$  may not always be feasible if $N $ is large and the sample size $n$ is not sufficiently large to ensure that, with high probability, the bootstrap estimate of the sample covariance matrix is not singular.

   Turning to the (local) intrinsic mean $\mu_I$  of $Q$, Theorem \ref{th-2.4} leads to the asymptotic chisquare distribution
   \begin{equation}
   \label{eq-5.2}
   n[\phi(\mu_n)-\phi(\mu_I)]'\hat{\Lambda}\tilde{\Sigma}^{-1}\hat{\Lambda}[\phi(\mu_n)-\phi(\mu_I)]\rightarrow \chi^2(d)
   \end{equation}
   in distribution as $n\rightarrow \infty$, where $\hat{}$ denotes an estimate with $Q$ replaced by the empirical $Q_n$; that is, the distribution $Q\circ\phi^{-1}$ of $Y_1$ is replaced by $Q_n\circ \phi^{-1}  = n^{-1}\sum_{1\leq i\leq n}\delta_{Y_i}$ . This leads to a confidence region for $\mu_I$. 
   

   We next consider the  the two-sample problem of distinguishing two distributions $Q_1 $ and $Q_2$ on $M$, based on two independent samples of sizes $n_1$ and $n_2$, respectively: $\{Y_{j_1} = J(X_{j_1}): j=1,\cdots,n_1\}, \{Y_{j_2} = J(X_{j_2}): j=1,\cdots,n_2\}$. Hence the proper null hypothesis is $H_0: Q_1 = Q_2.$ For high dimensional $M$ it is often sufficient to test if the two Fr\'echet means are equal. For the extrinsic procedure, again consider an embedding $J$ into $E^N$.  Write $\mu_i$  for $\mu_i^J $ for the population means and $\bar{Y}_i$ for the corresponding sample means on $E^N$ ($i=1,2)$. Let $n = n_1+n_2$, and assume $n_1/n\rightarrow p_1$, $n_2/n\rightarrow p_2 = 1-p_1$, $0<p_i<1 (i=1,2)$, as $n\rightarrow \infty$.  If $\mu_1\neq\mu_2$ then $Q_1 \neq Q_2$ .  One may then test $H_0:  \mu_1 =  \mu_2 (=\mu$, say).  Since $N$ is generally quite large compared to $d$, the direct test for $H_0:  \mu_1 =
   \mu_2 $ based on $\bar{Y}_1-  \bar{Y}_2$ is generally not a good test. Instead, we compare the two extrinsic means  $\mu_{E_1}$ and $\mu_{E_2}$ of $Q_1$ and  $Q_2$ and test for their equality. This is equivalent to testing if $P(\mu_1 )=P( \mu_2 )$.  Then, by \eqref{eq-5.1}, assuming  $H_0$,
   \begin{equation}
\label{eq-5.3} n^{1/2}  d_{\bar{Y}}P(\bar{Y}_1   - \bar{Y}_2 ) \rightarrow N(0,  B(p_1\Sigma_1 + p_2 \Sigma_2)B' )
   \end{equation}
     in distribution, as $n\rightarrow \infty$.
Here $\bar{Y} =  p_1\bar{Y}_1+p_2\bar{Y}_2$ is the \emph{pooled estimate} of the common mean $\mu_1=\mu_2=\mu$, say,  $B = B(\mu$) (see Proposition \ref{prop-4.3} ),  and $\Sigma_1$, $\Sigma_2$ are the covariance matrices of $Y_{j_1}$ and $Y_{j_2}$ . This leads to the asymptotic chisquare statistic below:
\begin{equation}
\label{eq-5.4}
  n[d_{\bar{Y}}P(\bar{Y}_1   - \bar{Y}_2  )]' [\hat{B}'(p_1\hat{\Sigma}_1+ p_2\hat{\Sigma}_2)\hat{B}  ]^{-1}[d_{\bar{Y}}P(\bar{Y}_1   - \bar{Y}_2  )]\rightarrow \chi^2_d
  \end{equation}
in distribution, as $n\rightarrow \infty$.
Here $\hat{B} = B(\bar{Y} )$, $\hat{\Sigma}_i$ is the sample covariance matrix of $Y_{ji}$. One rejects the null hypothesis $H_0$ at a level of significance $1-\alpha$ if and only if the observed value of the left side of \eqref{eq-5.4} exceeds $\chi^2_d(1-\alpha).$

  For the two-sample intrinsic test, let $\mu_{I_1}$, $\mu_{I_2}$ denote the intrinsic means of $Q_1$ and $Q_2$ and consider $H_0:  \mu_{I_1} =  \mu_{I_2}$. Denoting by $\mu_{n_1}$, $\mu_{n_2}$ the  intrinsic sample means, \eqref{eq-5.4} implies that, under $H_0$,
\begin{equation}
\label{eq-5.5}
 n^{1/2}[\phi(\mu_{n_1})- \phi(\mu_{n_2})] \rightarrow N(0, p_1 \Lambda_1^{-1}\tilde{\Sigma}_1 \Lambda_1^{-1}+ p_2 \Lambda_2^{-1}\tilde{\Sigma}_2 \Lambda_2^{-1} )
 \end{equation}
  in distribution,
where  $\phi = Exp_p^{-1}$ for some convenient $p$ in $M$, and $\Lambda_i$, $\tilde{\Sigma}_i$  are as in Theorem \ref{th-2.4} and \eqref{eq-5.4} with the empirical $Q_{n_i}$ in place of $Q_i$ $(i=1,2)$. One simple choice for $p$ is the pooled estimate $\mu_n = p_1 \mu_{n_1} + p_2  \mu_{n_2}$ , and with this choice we write $\hat{\phi}$ for $\phi$.  The test then rejects $H_0:  Q_1 =  Q_2$ , if
\begin{equation}
\label{eq-5.6}
  n[\hat{\phi}(\mu_{n_1})- \hat{\phi}(\mu_{n_2})]' [p_1\hat{\Lambda}_1^{-1}\hat{\tilde{\Sigma}}_1\hat{\Lambda}_1^{-1}+p_2\hat{\Lambda}_2^{-1}\hat{\tilde{\Sigma}}_2\hat{\Lambda}_2^{-1}]^{-1} [\hat{\phi}(\mu_{n_1})- \hat{\phi}(\mu_{n_2})]   >  \chi^2_d(1-\alpha).
\end{equation}

Finally,  consider a \emph{match pair problem } with i.i.d. observations $(X_{j_1},X_{j_2})$  having the distribution $Q$ on the product manifold $M\times M$. If $J $ is an embedding of $M$ into $E^N$, then $\tilde{J}(x,y) = (J(x),J(y))$ is an embedding of $M\times M$ into $E^N \times  E^N$. Let $\mu_{E_1}$, $\mu_{E_2}$  be the extrinsic means of the (marginal) distributions $Q_1$ and $Q_2$ of $X_{j_1}$ and $X_{j_2}$, respectively.  Once again, we are interested in testing $H_0: Q_1 = Q_2$  by checking if $\mu_{E_1}=\mu_{E_2}$ .  Note that the extrinsic mean of $Q$ is $\tilde{\mu}_E = (\mu_{E_1}, \mu_{E_2})$. If $\bar{Y}_1$ , $\bar{Y}_2$ are the sample means of $Y_{j_1} = J(X_{j_1})$, $Y_{j_2}= J(X_{j_2})$, $j=1,\cdots,n$, on $E^N$ with $E(Y_{j_1} )= \mu_1$ and $E(Y_{j_2} ) = \mu_2$, and  $\bar{\tilde{Y}} = (\bar{Y}_1 , \bar{Y}_2 )$,  then the extrinsic sample mean in the image space $\tilde{J}(M\times M)$ is $(P(\bar{Y}_1),P(\bar{Y}_2) )$.  Also, write $ \bar{Y}=  (\bar{Y}_1 + \bar{Y}_2 )/2$. Under $H_0$, $\mu_1=\mu_2=\mu$, say, and one has
\begin{equation}
\label{eq-5.7}
 n^{1/2}d_{\bar{Y}} P(\bar{Y}_1 - \bar{Y}_2 )  \rightarrow N(0, \Sigma_{11} +  \Sigma_{22} - \Sigma_{12} - \Sigma_{21} ).
 \end{equation}
On the right, $\Sigma_{11}$ and $\Sigma_{22}$ are the $d\times d$ covariance matrices of $(d_{\mu} P)(Y_{j_1}-\mu_1)$  and   $(d_{\mu} P)(Y_{j_2}-\mu_2)$,  while $\Sigma_{12}$ is the $d\times d$ cross covariance matrix of $(d_{\mu} P)(Y_{j_1}-\mu_1)$ and   $(d_{\mu} P)(Y_{j_2}-\mu_2)$,  and $\Sigma_{21}  = \Sigma_{12}'$ . As above, one  derives a chisquare test for $H_0 $, using \eqref{eq-5.7} and sample estimates of the covariance matrices.

\section{Intrinsic and Extrinsic Analysis and Curvature }

In this section  we provide explicit expressions of asymptotic dispersions for the intrinsic CLT on general Riemannian manifolds, and relate this to curvature, with applications to the sphere and planar shape spaces.

For intrinsic analysis, consider the function $h(z,y) = \rho_g^2(Exp_pz, Exp_py)$ for $z$, $y$ in $T_pM$, with an appropriate choice of $p$.  One first needs to express explicitly the quantities $D_rh(z,y)$, $D_rD_sh(z,y)$ in  normal coordinates at $p$, i.e., at $z =0\equiv Exp_p^{-1}p$.  For this let $\gamma(s)$ be a (constant speed) geodesic starting at $p$, and $m \in M$. Define the \emph{parametric surface}  $c(s,t)=Exp_m(tExp_m^{-1}\gamma(s))$, $s \in [0,\epsilon)$, $\epsilon> 0$ small. Note that  $c(s,0) = m$ for all $s$, $c(s,1) = \gamma(s)$, and that, for all fixed $s \in [0,\epsilon)$, $t\rightarrow c(s,t)$ is a geodesic starting at $m$ and reaching $\gamma(s)$ at $t=1$.  Writing $T(s,t) = (\partial/\partial t)c(s,t)$, $S(s,t) = (\partial/\partial s)c(s,t)$, one then has $S(s,0) = 0,$ $S(s,1) = \dot\gamma(s)$.  Also, $\langle T(s,t),T(s,t)\rangle$  does not depend on $t$ and, therefore,
\begin{equation}\label{eq-7.1}
 \rho_g^2 (\gamma(s),m) = \int_0^1 \langle T(s,t), T(s,t)\rangle dt.
 \end{equation}
Differentiating this with respect to $s$ and recalling the symmetry $(D/\partial s)T(s,t) = (D/\partial t)S(s,t$) on a parametric surface (See Do Carmo (1992), p. 68, Lemma 3.4), and $(D/\partial t) T(s,t) =0$, one has
\begin{align}\label{eq-7.2}
 (d/ds) \rho_g^2 (\gamma(s),m)& = 2\int_0^1 \langle(D/\partial s)T(s,t), T(s,t)\rangle dt\\ \nonumber
  &= 2\int_0^1 \langle(D/\partial t)S(s,t), T(s,t)\rangle dt=2\int_0^1 (d/dt)\langle S(s,t), T(s,t)\rangle dt\\ \nonumber
  & = 2\langle S(s,1),T(s,1)\rangle  =  -2\langle \dot\gamma(s),  Exp_{\gamma(s)}^{-1}m\rangle.
         \end{align}
Setting $s=0$ in \eqref{eq-7.2} and letting $\dot\gamma(0) = v_r$, with $\{v_r: r=1,\cdots,d\}$ an orthonormal basis of $T_pM$,  one shows that the normal coordinates $y_r$ of $m$ (i.e., the coordinates of $y = Exp_p^{-1}m$  with respect to $\{v_r: r=1,\cdots,d\})$ satisfy
\begin{equation}\label{eq-7.3}
 -2 y^r \equiv -2 \langle Exp_p^{-1}m, v_r\rangle = [(d/ds) \rho_g^2 (\gamma(s),m)]_{s=0}.
 \end{equation}
From this one gets
\begin{equation}
\label{eq-7.4}
 D_rh(0, y) = -2y^r (r=1,\cdots,d).
 \end{equation}
If $Q(Cut(p)) =0$, then writing $\tilde{Q}$ for the distribution induced from $Q$ by the map  $Exp_p^{-1}$ on $T_pM$, the Fr\'echet function and its \emph{gradient} in local coordinates  may be expressed as
\begin{equation}\label{eq-7.5}
F(q) = \int \rho_g^2 (q,m)Q(dm) = \int h(z,y)\tilde{Q}(dy) =\tilde{F}(z), 
\end{equation}
where $z =  Exp_p^{-1}q$ and $D_r\tilde F(z)=-2\int y^r\tilde Q(dy).$
Since a (local) minimum of this is attained at $q =\mu_I$, $\tilde{F}$ must satisfy a first order condition $D_r\tilde{F}(z) =0$ at $z=\nu$. In particular, letting $p= \mu_I$ and, consequently,  $\nu=0$, one has $\int D_rh(0,y)\tilde{Q}(dy) =0$, so that \eqref{eq-7.4} yields
\begin{equation}\label{eq-7.6}
 \int y^r\tilde{Q}(dy) =0 \;(r=1,\cdots,d),\;    (\tilde{Q}= Q\circ \phi^{-1}, \phi= Exp_{\mu_I}^{-1}).
 \end{equation}

By Theorem \ref{th-2.4}, the asymptotic distribution of the sample intrinsic mean $\mu_n$ is that of $\phi^{-1}(\nu_n)$, where $\phi= Exp_p^{-1}$, and
\begin{equation}\label{eq-7.7}
 \sqrt n(\nu_n  -\nu)  \simeq \Lambda^{-1} [(1/\sqrt n)\sum_{1\leq j\leq n} Dh(\nu, Y_j)],\;      ( \Lambda_{rs}= ED_rD_s h(\nu,Y_1), 1\leq r,s\leq d ),
 \end{equation}
with $Y_j =  \phi(X_j)$, where  $X_j$ are  i.i.d. with distribution $Q$.  By \eqref{eq-7.4}, the right side of \eqref{eq-7.7} simplifies to $\Lambda^{-1} [-2(1/\sqrt n)\sum_{1\leq j\leq n} Y_j]$, if $p =\mu_I$ (and $\nu=0$).

For an analytical study of the Hessian $\Lambda$ of the Fr\'echet function, one derives from \eqref{eq-7.14} the relation 
\begin{align}
\frac{d^2}{ds^2}\rho_g^2 (\gamma(s), m)=2\langle D_sT(s,1), S(s,1)\rangle=2\langle D_tS(s,1), S(s,1)\rangle,\\
\left(  D_s=\frac{D}{\partial s}, D_t=\frac{D}{\partial t} \; \text{covariant derivatives }\right).
\end{align}

Using the theory of Jacobi fields (\cite{docarmo}, p.111) the following relations may be derived. Let $C$ denote the supremum of all sectional curvatures of $M$ and let 
\begin{align}
f(t)=\begin{cases}
1& \text{if}\; C=0\\
(\sqrt{C}t)\cos(\sqrt{C}t)/\sin (\sqrt{C}t)&\text{if}\; C>0,\\
\sqrt{-C}t\cosh(\sqrt{-C}t)/\sinh (\sqrt{-C}t)& \text{if}\; C<0.
\end{cases}
\end{align}
Also let  $t_0$ be the supremum of all $t$ such taht $f(t)>0$. For $d\times d$ symmetric matrices $A$, $B$, the order relation $A\geq B$ means $A-B$ is nonnegative definite.

\begin{theorem}[\cite{absrabi1}]
\label{th-6.1}
Assume $ |Y_1|=|\log_{\mu_I}X_1|\leq t_0$ a.s. In addition, if the hypotheses for the CLT in corollary \ref{coro-1} or Theorem \ref{th-2.4} hold, one has
\begin{align}
\label{eq-lambda}
\Lambda=((\Lambda_{ij}))\geq \left(\left(   2E\left(  \frac{1-f|Y_1|}{|Y_1|^2}Y_1^iY_1^j+f(|Y_1|)\delta_{ij}\right)                 \right)\right),
\end{align}
with equality if the sectional curvature is consistent. 
\end{theorem}

\begin{remark}
It is simple to check that on $M=S^d$, the Hessian $\Lambda$ given by the right side of \eqref{eq-lambda}, with $C=1$, is nonsingular. 

\end{remark}

\begin{remark}
\cite{Kendall2011} obtained the exact expression for the Hessian for the intrinsic Fr\'echet mean on the important case of the planar shape space $\Sigma_2^k$, which has a constant holomorphic curvature.
\end{remark}

\begin{remark}
Note that the relations in \eqref{eq-7.3} provide the gradient of the intrinsic Fr\'echet function $F(p)$ in normal coordinates around $p$. 
\end{remark}

\textbf{Example 6.5.} (Confidence region for the intrinsic mean of $Q$ on the sphere $S^d$). Let $\mu_I$ be the intrinsic mean of $Q$ on $S^d$.  Given $n$ i.i.d. observations  $X_1,\cdots,X_n$ on $S^d$ with common distribution $Q$, let $\mu_n$ be the intrinsic sample mean. Write $\phi= Exp_{\mu_I}^{-1}$, and  $\phi_p= Exp_p^{-1}$, so that $\phi_{\mu_I}=\phi$.  By Theorem \ref{th-clt},
\begin{equation}\label{eq-7.8}
  \sqrt n[\phi(\mu_n)-\phi(\mu_I) )= \sqrt n\phi(\mu_n)\rightarrow N(0,  \Lambda^{-1}\tilde{\Sigma} \Lambda^{-1})\;\text{  in distribution as}\; n\rightarrow\infty,
  \end{equation}
where  the $d\times d $ matrices $\Lambda$ and $\tilde{\Sigma}$ are given by
\begin{align}\label{eq-7.9}
 \tilde{\Sigma} &= 4Cov (\phi(X_1)), \; \\ \nonumber
    \Lambda_{rs}&= 2E [ (1- (X_1^t\mu_I)^2)^{-1} \{1- (1- (X_1^t\mu_I)^2)^{-1/2}\cdot(X_1^t\mu_I)    \arccos(X_1^t\mu_I) \}( X_1^t \nu_r)(X_1^t \nu_s)\\ \nonumber
&+  (1- (X_1^t\mu_I)^2)^{-1/2}\cdot(X_1^t\mu_I) (\arccos(X_1^t\mu_I))) \delta_{rs}  ],   1\leq r,s \leq d.
\end{align}
Here  $\{  \nu_r:1\leq r\leq d\}$ is an orthonormal basis of $T_{\mu_I}S^d$.
A confidence region for $\mu_I$ , of asymptotic level $1-\alpha$, is then given by
\begin{equation}\label{eq-7.10}
\{p \in S^d:   n \phi_p(\mu_n)^t \hat{\Lambda}_p\hat{\tilde{\Sigma}}_p^{-1} \hat{\Lambda}_p\phi_p(\mu_n) \leq \chi^2_d(1-\alpha)\},
\end{equation}
where $\Lambda_p$, $\tilde{\Sigma}_p$ are obtained by replacing $\mu_I$ by $p$ in the expressions for $\Lambda$ and $\tilde{\Sigma}$ in \eqref{eq-7.9} . The  'hat' ( $\hat{}$ ) indicates that the expectations are computed under
the empirical $Q_n$, rather than $Q$.  As mentioned in Section 5, it would be computationally simpler to choose a particular $p =p_0$, say, and let $\phi= Exp_{p_0}^{-1}$. Then \eqref{eq-5.2} yields a simpler confidence region:
\begin{equation}\label{eq-7.11}
\{p \in S^d: n[\phi(\mu_n)-\phi(p)]^t \hat{\Lambda}_{p_0}\tilde{\Sigma}_{p_0}^{-1} \hat{\Lambda}_{p_0} [\phi(\mu_n) )-\phi(\mu_p)]\leq\chi^2_d(1-\alpha)\}.
\end{equation}

\textbf{Example 6.6.} (Inference on the planar shape space $\Sigma_2^k$).  
 

To apply Theorem \ref{th-2.4}, we use \eqref{eq-7.7} where $\phi= Exp _{\sigma(p)}^{-1}$ and $p$ is a suitable point in $\mathbb{C} S^{k-1}$. To derive a computable expression for $\Lambda$, write the geodesic $\gamma$ in the parametric surface $c(s,t)$ as $\gamma=\pi\circ\tilde{\gamma}$, where $\tilde{\gamma}$ is a geodesic in $\mathbb{C} S^{k-1}$   starting at $\tilde{\mu}\in \pi^{-1}\{\mu_I\}$. Then, with $\tilde{T}(s,1) = (d_{\gamma(s)}\pi^{-1} ) T(s,1)$,
\begin{align}
\label{eq-7.14}
 (d/ds) \rho_g^2 (\gamma(s),m)& = 2<T(s,1), \dot\gamma (s)> = 2<\tilde{T}(s,1),\dot{\tilde{\gamma}}(s)> ,\\ \nonumber
           (d^2/ds^2) \rho_g^2 (\gamma(s),m) &= 2<D_s\tilde{T}(s,1), \dot{\tilde{\gamma}}(s)>.
 \end{align}
The final inner products are in $T\mathbb{C} S^{k-1}$  , namely, $\langle \tilde{v},\tilde{w}\rangle = Re(\tilde{v}\tilde{w}^*)$.  Note that    $\tilde{T}(s,1) = -Exp_{\tilde{\gamma}(s)}^{-1} q, \;q\in\pi^{-1}{m}$,  may be expressed by \eqref{eq-6.5} and \eqref{eq-6.6}  as
\begin{equation}\label{eq-7.15}
 \tilde{T}(s,1) = -(\rho(s)/\sin \rho(s))[e^{i\theta(s)}q -(\cos \rho(s))\tilde{\gamma}(s)],
\end{equation}
where $\rho(s) = \rho_g(\gamma(s),m)$ and $e^{i\theta(s)} = (1/\cos \rho(s)) \tilde{\gamma}(s)q^*$ . The covariant derivative $D_s\tilde{T}(s,1)$ is the projection of $(d/ds)\tilde{T}(s,1)$ onto $H_{\tilde{\gamma}(s)}$. Since $\langle\tilde{\mu},\dot{\tilde{\gamma}}(0)\rangle =0$, \eqref{eq-7.14}  then yields
\begin{equation}\label{eq-7.16}
 [(d^2/ds^2) \rho_g^2 (\gamma(s),m)]_{s=0} = 2\langle[(d/ds)\tilde{T}(s,1)]_{s=0}, \dot{\tilde{\gamma}}(0)\rangle.
 \end{equation}
Differentiating \eqref{eq-7.15} one obtains
\begin{align}\label{eq-7.17}
 &[(d/ds)\tilde{T}(s,1)]_{s=0} = [(d/ds) (\rho(s)\cos \rho(s))/\sin \rho(s) )]_{s=0}\tilde{\mu}\\ \nonumber
  &+ [(\rho(s)cos \rho(s))/sin \rho(s) )]_{s=0} \dot{\tilde{\gamma}}(0) - [(d/ds) (\rho(s)/(\cos \rho(s))(\sin \rho(s))]_{s=0} (\tilde{\mu}q^*)q\\ \nonumber
  &- [\rho(s)/(\cos \rho(s))(\sin \rho(s))]_{s=0} (\dot{\tilde{\gamma}}(0)q*)q.
\end{align}
From \eqref{eq-7.14}, $2\rho(s)\dot\rho(s) = 2\langle\tilde{T}(s,1), \tilde{\gamma}'(s)\rangle$, which along with \eqref{eq-7.15} leads to
\begin{equation}\label{eq-7.18}
 [(d/ds)\rho(s)]_{s=0} = -(1/\sin r) \langle(\tilde{\mu}q^*/\cos r)q ,\dot{\tilde{\gamma}}(0) \rangle,\; (r= \rho_g(m,\mu_I) ).
 \end{equation}
One then gets (See Bhattacharya and Bhattacharya (2008), pp. 93, 94 )
\begin{align}\label{eq-7.19}
&\langle[(d/ds)\tilde{T}(s,1)]_{s=0} ,\dot{\tilde{\gamma}}(0)\rangle =  \{(r\cos r)/(\sin r)\}|\dot{\tilde{\gamma}}(0)|^2 \\ \nonumber
 &- \{(1/\sin^2 r) -(r\cos r)/\sin^3r\}(Re(x))^2  + r/((\sin r)(\cos r) ) (Im(x))^2,\\ \nonumber
         &(x = e^{i\theta}q\dot{\tilde{\gamma}}(0)^*,  e^{i\theta} =\tilde{\mu}q^*/\cos r ).
         \end{align}
One can check that the right side of \eqref{eq-7.19} depends only on $\pi(\tilde{\mu})$ and not any particular choice of $\tilde{\mu}$ in $\pi^{-1}\{\mu_I \}$.

Now let $\{\nu_1,\cdots,\nu_{k-2}, i\nu_1,\cdots,i\nu_{k-2}\}$ be an orthonormal basis of  $T_{\sigma(p)} \Sigma_2^k$ where we identify $\Sigma_2^k$  with $\mathbb{C}P^{k-2}$, and choose the unit vectors $\nu_r = (\nu_r^1,\cdots, \nu_r^{k-1})$, $r=1,\cdots,k-2$, to have  zero imaginary parts and satisfy the conditions  $p^*\nu_r =0,$ $\nu_r^t \nu_s =0$ for $r \neq s.$

  Suppose now that $\sigma(p)=\mu_I$, i.e., $\gamma(0) = \mu_I$. If $\dot\gamma(0) =v$, then $\gamma(s) = Exp_{\mu_I}(sv)$, so that $\rho_g^2 (\gamma(s),m)=h(sv, y)$ with $y = Exp_{\mu_I}^{-1} m$. Then, expressing $v$ in terms of the orthonormal basis,
\begin{equation}\label{eq-7.20}
  [(d^2/ds^2) \rho_g^2 (\gamma(s),m)]_{s=0} = [(d^2/ds^2)h(sv,y)]_{s=0} = \Sigma v_iv_jD_iD_jh(0,y).
  \end{equation}
Integrating with respect to $Q$ now yields
\begin{equation}\label{eq-7.21}
\sum v_iv_j  \Lambda_{ij} = E[(d^2/ds^2) \rho_g^2 (\gamma(s),X)]_{s=0} ,\;   \;(X\;\text{ with distribution}\; Q).
\end{equation}
This identifies the matrix $\Lambda$ from the calculations \eqref{eq-7.16} and \eqref{eq-7.19}.  To be specific, consider independent  observations $X_1,\cdots,X_n$ from $Q$, and let $Y_j= Exp_{\mu_I}^{-1}X_j (j=1,\cdots,n)$.  In normal coordinates with respect to the above basis of $T_{\mu_I}\Sigma_2^k$, one has the following coordinates of $Y_j$:
\begin{equation}\label{eq-7.22}
(Re(Y_{j}^1),\cdots, Re(Y_{j}^{k-2}), Im(Y_{j}^{1}),\cdots, Im(Y_{j}^{k-2}) ) \in \RR^{2k-4}.
\end{equation}
Writing

$$\Lambda=\begin{pmatrix}
 \Lambda_{11}     &        \Lambda_{12}\\

   \Lambda_{21}     &       \Lambda_{22}
 \end{pmatrix}$$
in blocks of $(k-2)\times(k-2)$ matrices, one arrives at the following expressions of the elements of these matrices, using \eqref{eq-7.19}- \eqref{eq-7.22}. Denote $\rho_g^2(\mu_I,X_1)=h(0,Y_1)$ by $\rho$. Then
\begin{align}\label{eq-7.23}
(\Lambda_{11})_{rs}&= 2E[\rho(\cot \rho)\delta_{rs}  -(1/\rho^2)( 1-\rho\cot \rho)(ReY_1^r)(ReY_1^s)\\ \nonumber
                          &+ \rho^{-1}(\tan \rho) (ImY_1^r)(ImY_1^s)];\\ \nonumber
(\Lambda_{22})_{rs}&=2E[\rho(\cot \rho)\delta_{rs}  -(1/\rho^2)( 1-\rho\cot \rho)(ImY_1^r)(ImY_1^s)\\ \nonumber
                          &+ \rho^{-1}(\tan \rho) (ReY_1^r)(ReY_1^s)];\\ \nonumber
(\Lambda_{12})_{rs}&=2E[\rho(\cot \rho)\delta_{rs}   -(1/\rho^2)( 1-\rho\cot \rho)(ReY_1^r)(ImY_1^s)\\ \nonumber
                          &+ \rho^{-1}(\tan \rho) (ImY_1^r)(ReY_1^s)]; \\ \nonumber
                           &(\Lambda_{21})_{rs} = (\Lambda_{12})_{sr}. (r,s =1,\cdots, k-2).
                          \end{align}
One now arrives at the  CLT for the intrinsic sample mean $\mu_n $ by Theorem \ref{th-2.4} and Corollary \ref{coro-1}.  A two-sample test for $H_0:$ $ Q_1=Q_2$, is then provided by \eqref{eq-5.2}.

  We next turn to extrinsic analysis on $\Sigma_2^k$, using the embedding \eqref{eq-6.9}.  Let $\mu_J$ be the mean of $Q\circ J^{-1}$ on $S(k-1, \mathbb{C})$, where $J$ is the veronese-Whitney embedding \eqref{eq-6.9}.

  Assuming that the largest eigenvalue of $\mu_J$ is simple (see proposition \ref{prop-3.1}), one may now obtain the asymptotic distribution of the sample extrinsic mean $\mu_{n,E}$, namely, that of $J(\mu_{n,E}) =v_n^*v_n$, where $v_n$ is a unit eigenvector of $\bar{\tilde{Y}}= \sum \tilde{Y}_j/n $  corresponding to its largest eigenvalue. Here $\tilde{Y}_j = J(Y_j)$, for i.i.d observations $Y_1,\cdots,Y_n$ on $\Sigma_2^k$. For this purpose, a convenient orthonormal basis (frame) of $T_p S(k-1, \mathbb{C})\simeq S(k-1, \mathbb{C})$ is the following:
\begin{align}\label{eq-7.25}
 \nu_{a,b}  & = 2^{-1/2}(e_ae_b^t+e_be_a^t )\;\text{ for}\; a<b,  \nu_{a,a}=e_ae_a^t;\\ \nonumber
              w_{a,b} &=i2^{-1/2}(e_ae_b^t-e_be_a^t )\;\text{ for}\; b<a\;   (a,b = 1,\cdots, k-1),
              \end{align}
where $e_a$ is the column vector with all entries zero other than the $a$-th, and the $a$-th entry is 1. Let $U_1,\cdots,  U_{k-1}$  be orthonormal unit eigenvectors corresponding to the eigenvalues $\lambda_1\leq\cdots\leq  \lambda_{k-2} < \lambda_{k-1}$.  Then choosing  $T = (U_1,\cdots, U_{k-1}   )\in SU(k-1)$  $T\mu_J T^*= D= diag(\lambda_1,\cdots,\lambda_{k-1})$, such that the columns of $T \nu_{a,b}   T^*$ and $Tw_{a,b} T^*$ together constitute an orthonormal basis of $S(k-1, \mathbb{C})$. It is not difficult to check that the differential of the projection operator $P$ satisfies

\begin{align}\label{eq-7.26}
(d_{\mu_J}P)Tv_{a,b}T^* =\begin{cases} 0&\text{ if}\; 1\leq a\leq b <k-1,\;\text{ or}\; a=b=k-1,\\
                                  ( \lambda_{k-1}-\lambda_a)^{-1}Tv_{a,k-1}T^* &\text{if}\; 1\leq a<k-1, b=k-1;
                                   \end{cases}
                                   \end{align}

\begin{align*}
(d_{\mu_J}P)Tw_{a,b}T^*= \begin{cases}0 & \text{if}\; 1\leq a\leq b <k-1,\\
                                     (\lambda_{k-1}-\lambda_a)^{-1}Tw_{a,k-1}T^* &\text{ if}\; 1\leq a <k-1.
                                      \end{cases}\end{align*}

To check these, take the projection of a linear curve $c(s$) in $S(k-1, \mathbb{C})$ such that $\dot c(0)$ is one of the basis elements $v_{a,b}$, or $w_{a,b}$, and differentiate the projected curve with respect to $s$. It follows that $\{ Tv_{a,k-1}T^*,  Tw_{a,k-1}T^* : a=1,\cdots, k-2\}$ form an orthonotmal basis of $T_{P(\mu_J)}J(\Sigma_2^k $).   Expressing $\tilde{Y}_j-\mu_J $ in the orthonormal basis of     $S(k-1, \mathbb{C}$),  and $d_{\mu_J}P(\tilde{Y}_j-\mu_J)$ with respect to the above basis of $T_{P(\mu_J)}J(\Sigma_2^k $), one may now apply Proposition \ref{prop-4.3}.

   For a two-sample test for $H_0: Q_1 =Q_2$, one may use \eqref{eq-5.4}, as explained in Section 5.

\section{Nonparametric Bayes Estimation of Densities on a Manifold and the Problem of Classification.}


\subsection{Density estimation}

Consider the problem of estimating the density $q$ of a distribution $Q$ on a  Riemannian manifold $(M,g)$ with respect to the volume measure $\lambda$ on $M.$  According to \cite{ferg73},  given a finite non-zero base measure $\alpha$ on a measurable space $(\mathcal X,\Sigma),$ a random probability $P$ on the class $\mathcal P$ of all probability measures on $\mathcal X $ has the \emph{Dirichlet distribution} $D_\alpha$ if for every measurable partition $\{B_1,\dots, B_k\}$ of $\mathcal X,$ the $D_\alpha$ - distribution of
$(P(B_1),\dots, P(B_k)) = (\theta_1,\dots,\theta_k),$ say, is Dirichlet with parameters
$(\alpha(B_1),\dots, \alpha(B_k)).$ That is, $\left(P(B_1),\ldots, P(B_{k-1}\right)$ has the joint density $f(\theta_1,\ldots, \theta_{k-1})=\text{const}\left(\theta_1^{\alpha(B_1)-1}\ldots\theta_{k-1}^{\alpha(B_k-1)-1} \right)\left(1-\theta_1-\ldots-\theta_{k-1}\right)^{\alpha(B_k)-1}$ on $\{(\theta_1,\ldots, \theta_{k-1}): \theta_i>0 \forall i, \; \theta_1+\ldots\theta_{k-1}<1\}$. If $\alpha(B)=0$ for some $B_j$, then $P(B_j)=0$ with probability 1. In the case $k=2$, the $D_{\alpha}$-distribution of $(P(B_1),P(B_2))$ is also called the \emph{beta distribution}, denoted $\text{beta}(\alpha(B_1),\alpha(B_2))$.
\cite{sethu} gave a very convenient ``stick breaking" representation of the random $P.$ To define it, let $u_j (j=1,\dots)$ be an i.i.d. sequence  of $beta (1, \alpha(\mathcal X))$ random variables, independent of a sequence $Y_j(j=1,\dots)$ having the distribution $G = \frac{\alpha}{\alpha(\mathcal X)}$ on $\mathcal X.$  Sethuraman's representation of the random probability with the Dirichlet prior distribution $D_\alpha$ is
\begin{equation}
\label{sethu}
P \equiv \sum w_j \delta_{Y_j} ,
\end{equation}
where $w_1 =u_1, w_j = u_j(1-u_1)\dots(1-u_{j-1}) (j=2,\dots),$ and $\delta_{Y_j}$ denotes the Dirac measure at $Y_j.$ As this construction shows, the Dirichlet distribution assigns probability one to the set of all discrete distributions on $\mathcal X,$ and one cannot retrieve a density estimate from it directly. The Dirichlet priors constitute a conjugate family, i.e., the posterior distribution of a random $P$ with distribution $D_\alpha,$ given observations $X_1,\dots,X_n$ from $P$ is $D_{\alpha +\sum_{1 \le i \le n} \delta_{X_i}} .$ A general method for Bayesian density estimation on a manifold $(M,g)$ may be outlined as follows. Suppose that $q$ is continuous and positive on $M.$ First find a parametric family of densities $m\to K(m;\mu,\tau)$ on $M$ where $\mu \in M$ and $\tau >0$ are ``location" and ``scale" parameters, such that $K$ is continuous in its arguments, $K(\cdot;\mu,\tau)d\lambda(\cdot)$ converges to $\delta_\mu$ as $\tau \downarrow 0,$ and the set of all ``mixtures" of $K(\cdot;\mu,\tau)$ by distributions on $M\times (0,\infty)$ is
dense in the set $C_\lambda(M)$ of all continuous densities on M in the supremum distance, or in $L^1(d\lambda).$ The density $q$ may then be estimated by a suitable mixture. To estimate the mixture, use a prior $D_\beta$ with full support on the set of all probabilities on the space $M\times (0,\infty)$ of ``parameters" $(\mu,\tau).$ A draw from the prior may be expressed in the form \eqref{sethu}, where $u_j$ are i.i.d. $beta (1, b)$ with
$b = \beta(M\times(0,\infty)),$ independent of $Y_j =(m_j,t_j),$ say, which are i.i.d. $\beta\over b$ on
$M\times (0,\infty).$ The corresponding random density is then obtained by integrating the kernel $K$ with respect to this random mixture distribution,
\begin{equation}
\label{mixt}
\sum w_jK(m; m_j,t_j).
\end{equation}

Given $M$-valued ($Q$-distributed) observations $X_1,\dots X_n,$ the posterior distribution of the mixture measure is Dirichlet $D_{\beta_X} ,$ where $\beta_X =\beta + \sum_{1\le i\le n}\delta_{Z_i} ,$ with $Z_i = (X_i,0).$ A draw from the posterior distribution leads to the random density in the form \eqref{mixt}, where $u_j$ are i.i.d. $beta (1, b+n),$ independent of $(m_j,t_j)$ which are i.i.d. $\beta_X\over{(b+n)}.$  One may also consider using a somewhat different type of priors such as  $D_\alpha \times \pi$ where $D_\alpha$ is a Dirichlet prior on $M,$ and $\pi$ is a prior on $(0,\infty),$ e.g., gamma or Weibull  distribution.

Consistency (weak consistency) of the posterior is generally established by checking full Kullback-Liebler support of the prior $D_\beta$ (See \cite{ghosh2003bayesian}, pp. 137-139). Strong consistency has been established for the planar shape spaces using the complex Watson family of densities (with respect to the volume measure or the uniform distribution on $\Sigma_2^k$) of the form $K([z];\mu,\tau) = c(\tau)exp{|z*\mu|^2\over\tau}$ in \cite{rabibook} and \cite{david1}, where it has been shown, by simulation from known distributions, that, based on a prior $D_\beta \times \pi$ chosen so as to produce clusters close to the support of the observations,  the Bayes estimates of quantiles and other indices far outperform the kernel density estimates (KDE) of \cite{bruno}, and also require much less computational  time than the latter. In moderate sample sizes, the nonparametric Bayes estimates perform much better than even the MLE (computed under the true model specification)!

\subsection{Classification.}
Classification of a random observation to one of several groups is one of the most important problems in statistics. This is the objective in medical diagnostics, classification of subspecies and, more generally, this is the target of most image analysis. Suppose there are r groups or populations with a priori given relative sizes or proportions $\pi_i (i=1, \dots,r), \sum \pi_i =1,$ and densities $q_i(x)$ (with respect to some sigma-finite measure). Under $0-1$ loss function, the average risk of misclassification (i.e., the Bayes risk) is minimized by the rule: Given a random observation $X$, classify it to belong to group $j$ if
$\pi_j q_j(X) = \max\{\pi_i q_i(X): i = 1,\dots,r\}$. Generally, one uses sample estimates of $\pi_i$-s and $q_i$ - s, based on random samples from the $r$ groups (training data). Nonparametric Bayes estimates of $q_i$-s on shapes spaces perform very well in classification of shapes, and occasionally identify outliers and misclassified observations (See, \cite{rabibook} and \cite{david1}).

 In a simulation study using 20 random draws from a complex Watson distribution, \cite{david1} found that the nonparametric Bayes estimate far outperformed the kernel density estimate KDE over a multitude of criteria.  It also performed much better than the MLE of the correctly specified model!  Here are the $L^1$ distances and the Kullback-Leibler divergences from $f_0$.
\begin{table}[h]
    \caption{$L^1$ Distance and Kullback-Leibler Divergence Between the Estimate and the True Density }
      \label{tb-NPbayes}
  \centering
\begin{tabular}{|p{4em}|p{5em}p{4em}p{4em}|}
  \hline
  & NP Bayes       &         KDE        &             MLE\\
  \hline
   $ L^1$     &              0.44      &                1.03   &                   0.75\\
   K-L     &             0.13         &           0.41  &                     0.25 \\
   \hline
\end{tabular}

\end{table}

\section{The Laplace-Beltrami Operator in Machine Vision}

Mark Kac asked in a paper in 1966 in the American Mathematical Monthly : ``Can one hear the shape of a drum?". In other words, by listening to the frequencies of vibrations of a clamped drum, given by the eigenvalues of the Laplacian with Dirichlet (or zero) boundary condition, is it possible to reconstruct or identify the geometric shape of the drum?  The origin of this question may be traced back to Hermann Weyl's famous formula \citep{Weyl1911} : For any bounded domain $\Omega$ in $\mathbb \R^d$ with a smooth boundary, the number $N(\lambda)$ of eigenvalues of  the  (negative) Laplacian $-\bigtriangleup$  which are less than $\lambda$ has the asymptotic relation
          \begin{align}
          \label{eq-spectrum}
          N(\lambda) \sim \omega_d(2\pi)^{-d} \lambda^{d/2} \text{vol}(\Omega) \;\text{ as}\;  \lambda\rightarrow\infty\;\;           (\omega_d= \text{vol of unit ball in}\; \mathbb{R}^d).
          \end{align}
Here the relation $\sim$ indicates that the ratio of its two sides converges to 1 (as $\lambda \rightarrow \infty$).  A similar formula holds for any $d$-dimensional compact Riemannian manifold $(M,g)$ with or without boundary where $\bigtriangleup$
is the so-called \emph{Laplace Beltrami operator} \citep{chavel1984eigenvalues, rosenberg1997laplacian}), which may be expressed in a local chart given by $u$ (on $B(0,r) \rightarrow U \subset M)$ as
          \begin{align}
          \label{eq-laplace}
            \bigtriangleup  f = (det\; g)^{-1/2}\sum_{1\leq i,j\leq d} \partial_ig^{ij}(det\; g)^{1/2}\partial_jf.
            \end{align}
Weyl type spectral asymptotics for $-\bigtriangleup$ are given by
     \begin{align}
     \label{eq-112}
     N(\lambda) \sim c(d) \lambda^{d/2} \text{vol}(\Omega) \; as \; \lambda\rightarrow \infty, \;
     \text{where}\; c(d) \;\text{depends only on the dimension}\; d.
     \end{align}
There are many refinements of the estimates \eqref{eq-spectrum}, \eqref{eq-11} with an error term of the order $\lambda^{(d-1)/2}$.  Although there are many spectral invariants of the manifold, it turns out unfortunately, that the answer to Kac's question is ``no" \citep{Milnor01041964}. For two-dimensional surfaces, the answer is mostly ``yes" outside a relatively small set of manifolds \citep{Zelditch2000}. But in dimension 3 or higher the set of non-isometric manifolds with the same spectrum is not negligible. A natural question that arises is: if one uses eigenfunctions as well as eigenvalues of $-\bigtriangleup$ can one reconstruct or identify the manifold?  One may find many interesting and important articles in computer science/machine vision journals where such reconstructions are  displayed.  But the mathematical question posed above is rigorously answered in the affirmative only by \cite{Jones08}, under only mild conditions on the manifold, such as uniform ellipticity of the Laplacian.  The last mentioned authors actually construct coordinate patches covering $M$, and therefore the structure of the manifold, only using eigenvalues and eigenfunctions of $-\bigtriangleup$.

There are many numerical methods for the computation of eigenvalues and eigenfunctions carried out by computer scientists and applied for object identification and scene recognition ( See, e.g. \cite{reuter06}).  To economize the use of these features sometimes topological properties of $M$ are also used. For example, see \cite{Dey:2008}, \cite{dey09} who use the first homology group to identify handles and holes in a closed bounded domain in 3D. For a more elaborate technique using algebraic topology, known as persistent homology, we refer to \cite{Carlsson09}.

\section{Examples and Applications}
\label{sec-ex}

In this section we apply the theory to a number of data sets available in the literature.

\textbf{Example 9.1.}  (Paleomagnetism). The first statistical confirmation of the shifting of the earth's magnetic poles over geological times, theorized by paleontologists based on observed fossilised magnetic rock samples, came in a seminal paper by R.A. Fisher (1953).  Fisher analyzed two sets of data - one recent (1947-48) and another old (Quaternary period),  using the so-called \emph{von Mises-Fisher model}
\begin{equation}
\label{eq-8.1} f(x; \mu, \tau) = c(\tau)\exp\{\tau x^t\mu\}   (x\in S^2),
\end{equation}
Here  $\mu (\in S^2)$, is the \emph{mean direction}, extrinsic as well as intrinsic ($\mu=\mu_I=\mu_E)$, and $\tau>0$ is the concentration parameter. The maximum likelihood estimate of $\mu$ is $\hat{\mu} =\bar{X}/| \bar{X}|$, which is also our sample extrinsic mean. The value of the MLE for the first data set of $n=9$ observations turned out to be $\hat{\mu} = \hat{\mu}_E=(.2984,.1346,.9449)$, where (0,0,1) is the geographic north pole.  Fisher's $95\%$ confidence region for $\mu$ is $\{\mu \in S^2: \rho_g(\hat{\mu}, \mu) \leq 0.1536)\}$.  The sample intrinsic mean is $\hat{\mu}_I= (.2990,.1349,.9447)$, which is very close to $\hat{\mu}_E$.The nonparametric confidence region based on $\hat{\mu}_I$, as given by  \eqref{eq-7.10},  and that based on the extrinsic procedure \eqref{eq-5.4}, are nearly the same, and both are about $10\%$ smaller in area than Fisher's region. (See \cite{rabibook}, Chapter 2).

The second data set based on $n=29$ observations from the Quaternary period that Fisher analyzed, using the same parametric model as above, had the MLE $\hat{\mu} = \bar{X}/| \bar{X}|= (.0172, -.2978, -.9545)$, almost antipodal of that for the first data set, and with a confidence region of geodesic radius .1475 around the MLE. Note that the two confidence regions are not only disjoint, they also lie far away from each other. This provided the first statistical confirmation of the hypothesis of shifts in the earth's magnetic poles, a result hailed by paleontologists (See \cite{Irvingbook}). Because of difficulty in accessing the second data set, the nonparametric procedures could not be applied to it.  But the analysis of another data set dating from the Jurassic period, with $n=33$,  once again yielded nonparametric intrinsic and extrinsic confidence regions very close to each other, and each about $10\%$ smaller than the region obtained by Fisher's parametric method (See \cite{rabibook}, Chapter 5, for details).

\textbf{Example 9.2.} (Brain scan of schizophrenic and normal patients). We consider an example from Bookstein (1991) in which 13 landmarks were recorded on a midsagittal two-dimensional slice from magnetic brain scans of each of 14 schizophrenic patients and 14 normal patients.  The object is to detect the   deformation, if any, in the shape of the $k$-ad due to the disease, and to use it for diagnostic purposes. The shape space is $\Sigma_2^{13}.$  The intrinsic two-sample test  \eqref{eq-5.6} has an observed value 95.4587 of the asymptotic chisquare statistic with 22 degrees of freedom, and a $p$-value  $3.97\times 10^{-11}$. The extrinsic test based on \eqref{eq-5.6} has an observed value 95.5476 of the chisquare statistic and a p-value $3.8\times10^{-11}$.  
 It is remarkable, and reassuring, that completely different methodologies of intrinsic and extrinsic inference essentially led to the same values of the corresponding asymptotic chisquare statistics (a phenomenon observed in other examples as well). For details of these calculations and others we refer to \cite{rabibook}. This may also be contrasted with the results of parametric inference in the literature for the same data, as may be found in \cite{dimk},  pp. 146, 162-165. Using a isotropic Normal model for the original landmarks data, and after  removal of ``nuisance" parameters for translation, size and rotation, an $F$-test known as Goodall's $F$-test (See \cite{bf}) gives a $p$-value .01.  A Monte Carlo test based permutation test obtained by 999 random assignments of the data into two groups and computing Goodall's $F$-statistic, gave a $p$-value .04. A Hotelling's $T^2$ test in the tangent space of the pooled sample mean had a $p$-value .834. A likelihood ratio test based on the isotropic offset Normal distribution on the shape space has the value 43.124 of the chisquare statistic with 22 degrees of freedom, and a $p$-value .005.

\textbf{Example 9.3.}  (Shapes of Gorilla Skulls)


We consider another example in which  two planar shape distributions via their extrinsic (and intrinsic) means  are distinguished. A  Bayesian nonparametric classifier is also  built and applied.

In this data set, there are 29 male and 30 female gorillas and the eight landmarks are chosen on the midline plane of the 2D image of the skull. The data can be found in \cite{dimk}. It is of interest to study the shapes of the skulls and use that to detect differences in shapes between the sexes. This finds applications in morphometrics and other biological sciences.



%

To distinguish between  the distribution of shapes of skulls of the two sexes, one may compare the sample extrinsic mean shapes or dispersions in shape as well as the intrinsic couterparts. 

The value of the two sample test statistic defined in \eqref{eq-5.6},  for comparing the intrinsic mean shapes, and the asymptotic p-value for the chi-squared test are
\begin{align*}
 T_{n1} =  391.63,\t{ p-value } = P(\mathcal{X}^2_{12} > 391.63) < 10^{-16}.
\end{align*}
Hence we reject the null hypothesis that the two sexes have the same intrinsic mean shape.
The  test statistics, defined in equations \eqref{eq-5.4} for comparing the extrinsic mean shapes, and the corresponding asymptotic p-values are
\begin{align*}
 T_1 &= 392.6, \t{ p-value } = P(\mathcal{X}^2_{12} > 392.6) < 10^{-16}.
\end{align*}
Hence we reject the null hypothesis that the two sexes have the same extrinsic mean shape. We can also compare the mean shapes by pivotal bootstrap method using the test statistic $T_2^*$ which is a bootstrap version of $T_2$. The p-value for the bootstrap test using  $10^5$ simulations turns out to be 0. In contrast, a parametric test carried out in \cite{dimk}, pp. 168-172, has a p-value .0001.

%
%

\begin{table}
\begin{minipage}{188pt}
\caption{Posterior probability of being female for each gorilla in the test sample.}
\label{egtab3}
\begin{tabular}{ccccc}
\\
gender &$\hat{p}([z])$ &95\% CI &$d_E([z_i],\hat\mu_1)$ &$d_E([z_i],\hat\mu_2)$ \\
F           &    1.000    &(1.000,1.000)     &0.041       &0.111         \\
F           &    1.000    &(0.999,1.000)       &0.036      &   0.093         \\
F           &    0.023   &(0.021, 0.678)     &0.056       &   0.052         \\
F           &    0.998    & (0.987, 1.000)  &0.050      &   0.095          \\
F           &   1.000    & (1.000, 1.000)     &0.076      &   0.135          \\
M           &   0.000     & (0.000, 0.000)     &0.167      &   0.103          \\
M           &   0.001     & (0.000, 0.004)  &0.087       &   0.042          \\
M           &   0.992     & (0.934, 1.000) &0.091      &   0.121          \\
M           &   0.000     & (0.000, 0.000)     &0.152      &   0.094          \\
\end{tabular}
$\hat{p}([z])$ = estimated prob. of being female, given shape $[z]$;
$d_E([z],\hat\mu_i)$ = extrinsic distance from the mean shape in group $i$, with $i=1$ for females and
$i=2$ for males
\end{minipage}
\end{table}


A Bayesian nonparametric classifier is next applied (see \cite{david1}) to predict gender. The shape densities for the two groups via non-parametric Bayesian methods are estimated which are used to derive the conditional distribution of gender given shape.  25 individuals of each gender are picked as a training sample, with the
remaining 9 used as test data. Table~\ref{egtab3} presents the estimated posterior probabilities of being female for each of the gorillas in the test sample along with a 95\% credible interval. For most of the gorillas, there is a high posterior probability of assigning the correct gender.  There is misclassification only in the 3rd female and 3rd male. For the 3rd female, the credible interval includes 0.5, suggesting that there is insufficient information to be confident in the classification.  However, for the 3rd male, the credible interval suggests a high degree of confidence that this individual is female.  Perhaps this individual is an outlier and there is something unusual about the shape of his skull, with such characteristics not represented in the training data, or, alternatively, he was labeled incorrectly.

\textbf{Example 9.4} (Corpus Callosum shapes of normal and ADHD children)

We consider the third planar  shape data set,  which involve measurements of a group typically developing children and a group of children suffering the ADHD (Attention deficit hyperactivity disorder).  ADHD  is one of the most common psychiatric  disorders for children that can continue through adolescence and adulthood. Symptoms include difficulty staying focused and paying attention, difficulty controlling behavior, and hyperactivity (over-activity). ADHD  in general has three subtypes: (1) ADHD hyperactive-impulsive (2) ADHD-inattentive; (3) Combined hyperactive-impulsive and inattentive  (ADHD-combined) \cite{ADHDtype}.  ADHD-200 Dataset (\url{http://fcon_1000.projects.nitrc.org/indi/adhd200/}) is a data set that  record both anatomical and resting-state functional MRI data of 776 labeled subjects across 8 independent imaging sites, 491 of which were obtained from typically developing individuals and 285 in children and adolescents with ADHD (ages: 7-21 years old).  
 The  Corpus Callosum shape data are extracted  using the CCSeg package, which contains 50 landmarks on the contour of the Corpus Callosum of each subject (see \cite{hongtu15}). 
 After quality control,  647 CC shape data out of 776 subjects were obtained, which included 404 ($n_1$) typically developing children, 150 ($n_2$) diagnosed with ADHD-Combined, 8 ($n_3$) diagnosed with ADHD-Hyperactive-Impulsive,  and 85 ($n_4$) diagnosed with ADHD-Inattentive. Therefore, the data lie in the space $\Sigma_2^{50}$, which has a high dimension of $2\times 50-4=96$.   
 

    We carry out  \emph{extrinsic two sample tests}  between the group of typically developing children and the group of children diagnosed with ADHD-Combined, and also between the group of typically developing children  and ADHD-Inattentive children. We construct test statistics that  base on the asymptotic distribution of the extrinsic mean for the planar shapes.
    
%

The $p$-value for the  two-sample test between the group of typically developing children and the group of children diagnosed with ADHD-Combined is $5.1988\times 10^{-11}$, which is based on the asymptotic chi-squared distribution given in \eqref{eq-5.4}. The $p$-value for the test between the group of typically developing children and the group  ADHD-Inattentive children is smaller than $10^{-50}$.  


\textbf{Example 9.5} (Positive definite matrices with application to diffusion tensor imaging.)

Another important class of manifolds is  $\sym^{+}(p)$, the space of $p\times p$ positive definite matrices. In particular, when $p=3$, $\sym^{+}(3)$, the space of  $3\times 3$  positive definite matrices, has important applications  in diffusion tensor imaging (DTI). DTI, is now an important tool for neuroimaging in clinical trials. It provides for the measurement of the diffusion matrix  (3 $\times$ 3 positive definite matrice) of molecules of water in tiny voxels in the white matter of the brain. When there are no barriers, the diffusion matrix is isotropic, and in the presence of structural barriers in the brain white matter due to axon (nerve fiber) bundles and their myelin sheaths (electrically insulating layers) the diffusion is anisotropic, and DTI can be used to measure the anisotropic diffusion tensor. When a trauma occurs, due to an injury or a disease, this highly organized structure is disrupted and anisotropy decreases. Large scale DTI based studies have been used to investigate autism, schizophrenia, Parkinson's disease and Alzheimer's disease. The geometry of  $\sym^{+}(p)$ for general $p$ is now described in the following.


Let $A\in \sym^{+}(p)$ which  follows a distribution $Q$.  We first introduce the Euclidean metric of $A$, which is given by $\|A\|^2=\trace(A)^2$.  Since $\sym^{+}(p)$ is an open convex subset of $\sym(p)$, the space of all $p\times p$ symmetric matrices, the mean  of $Q$  with respect to the Euclidean distance is given by the  Euclidean mean
\begin{equation}
\mu_E=\int A Q(dA).
\end{equation}

 Another  important metric for $\sym^{+}(p)$  is the \emph{$\log$-Euclidean metric} \citep{Arsigny06}. Let $J\equiv \log:  \sym^{+}(p)\rightarrow \sym(p)$ be the inverse of the \emph{exponential map} $B\rightarrow e^B$,  $\sym(p)\rightarrow \sym^{+}(p)$, which is the matrix exponential of $B$. Then $J$ is a diffeomorphism. The $\log$ Euclidean distance is given by
\begin{equation}
\rho_{LE}(A_1,A_2)=\| \log(A_1)-\log(A_2)\|.
\end{equation}
Note that $J$ is an embedding on  $\sym^{+}(p)$ onto $\sym (p)$ and, in fact, it is an equivariant embedding under the group action of $\text{GL}(p, \R)$ , the  general linear group of $p\times p$ non-singular matrices. The extrinsic mean of $Q$ under $J$ is given by
\begin{equation}
\mu_{E,J}=\exp(\int (\log (A))Q(dA)).
\end{equation}

We  apply Theorem \ref{th-clt} to  sample Fr\'echet means under both the Euclidean and $\log$-Euclidean distances. In particular, we  consider a diffusion tensor imaging (DTI) data set  consisting of 46 subjects with 28 HIV+ subjects and 18 healthy controls.   Diffusion tensors were extracted along the fiber tract of the splenium of the corpus callosum. The DTI data for all the subjects are  registered in the same \emph{atlas space} based on arc lengths, with 75 features  obtained along the fiber tract of each subject. This data set has been studied in a regression setting in \cite{Yuan2012}. 
On the other hand, we carry out two sample tests between the control group and the HIV+ group for each of the 75 sample points along the fiber tract. Therefore, 75 tests are performed  in total. Two types of tests are carried out based on the Euclidean distance and the log-Euclidean distance.  


The simple Bonferroni procedure for testing $H_0$ yields a $p$-value equal to 75 times the smallest $p$-value which is of order $10^{-7}$. To identify sites with significant differences, the 75 $p$-values are ordered from the smallest to the largest with a \emph{false discovery rate} of $\alpha=0.05$, $58$ sites are found to yield significant differences using the Euclidean distance, and 47 using the $\log$-Euclidean distance (see \cite{citeulike:1042553}).

 \textbf{ Example 9.6.} (Glaucoma detection- a match pair problem in $3D$). Our final example is on the $3D$ reflection similarity shape space $R\Sigma_3^k$. To detect shape changes due to glaucoma, data were collected on twelve mature rhesus monkeys.
One of the eyes of each monkey was treated with a chemical agent to temporarily increase the intraocular pressure (IOP). The increase in IOP is known to be a cause of glaucoma. The other eye was left untreated. Measurements were made of five landmarks in each eye, suggested by medical professionals. The data may be found in \cite{rabi05}.  The match pair test based on \eqref{eq-5.7}  yielded an observed value 36.29 of the asymptotic chisquare statistic with degrees of freedom 8. The corresponding $p$-value is $1.55\times10^{-5}$ (See \cite{rabibook}, Chapter 9). This provides a strong justification for using shape change of the inner eye as a diagnostic tool to detect the onset of glaucoma. An earlier computation using a different nonparametric procedure in \cite{rabi05} provided a $p$-value .058. Also see Bandulasiri et al. (2009) where a $95\%$ confidence region is obtained for the difference between the extrinsic size-and-shape reelection shapes between the treated and untreated eyes.

\appendix

\section{Appendix on Riemannian Manifolds}

\numberwithin{equation}{section}
\setcounter{equation}{0}

\label{sec-app}

Often the manifold $M$ in applications has a natural Riemannian metric tensor $g$. That is, it is given an inner product $\langle, \rangle_p$ on the tangent space $T_pM$ at $p$, which is smoothly defined.  In local coordinates in $U_p$ given by $\psi_p(\cdot)=x= (x_1,\ldots,x_d) \in B_p$, the functions $ (g_{ij})(x) = \langle E_i,E_j\rangle_p$ , with $E_i= d\psi_p^{01} (\partial /\partial x_i)$ ($i,j =1,\ldots,d$), are smooth in $B_p$. This allows one to measure the length of a smooth arc $\gamma$ joining any two points $q$, $q' $in $U_p$, namely,  $\int_{[a,b]}|dx(t)/dt| dt$,
$\gamma(a) = q$,  $\gamma(b) = q',$  $x(t) = \psi_p\circ \gamma(t)$. Here $|dx(t)/dt|^2= \langle dx(t)/dt, dx(t)/dt\rangle_p$, with $dx(t)/dt$ expressed in the local frame $E_i $($i=1,\ldots, d$).  One may also write $dx(t)/dt$ as $d\gamma(t)/ dt$. Using the compatibility condition (ii) above one now defines the length of a smooth arc joining any two points in $M$. The \emph{geodesic distance} $\rho_g(p,q)$ between $p$ and $q$ is the minimum of lengths of all smooth arcs joining $p$ and $q$.  A standard parametrization of a curve is its arc length $s$: $s=\int_{[a,t]} |d\gamma(u)/du| du$. In this parametrization of curves, one has $|d\gamma(t)/dt| =1$.  We will adopt this so called \emph{unit speed} parametrization unless otherwise specified. The property of local minimization of arc lengths yields a first order condition on the velocity $d\gamma(t)/dt$ of the minimizing curve $\gamma$ at $t$:  the acceleration along $\gamma$ is zero at every parameter join $t$. If $M$ is a submanifold ((hyper) surface) of an Euclidean space $\R^N$, then the second derivative $d^2\gamma(t)/dt^2$  is well defined,  but in general does not belong to the tangent space of $M$ at $\gamma(t)$.  By \emph{`acceleration'} one means the orthogonal projection of the vector $d^2\gamma(t)/dt^2$  onto the tangent space of $M$ at $\gamma(t)$.  This projection is called the \emph{covariant derivative} of the velocity and denoted $(D/dt) d\gamma(t)/dt$. The ``zero acceleration" of a geodesic $\gamma$ means $(D/dt) d\gamma(t)/dt=0$.  On a general differentiable manifold, which is not given explicitly as a submanifold, there is no ``outside".  The proper extension of the above notion of covariant derivative by Levi-Civita, using a notion known as \emph{affine connection}, for all differentiable manifolds was a milestone in the development of differential geometry (See, e.g.,  \cite{docarmo} Chapter 2).

In local coordinates the equation for a geodesic is a second order ordinary differential equation. By the standard existence theorem for ordinary differential equations, a geodesic $\gamma$ is uniquely determined on a maximal interval $(a, b)$ ($-\infty \leq a <b<\infty$), given an initial point $\gamma(0) =p$ and an velocity $(d\gamma(t)/dt)_{t=0} = v$. According to a result of Hopf and Rinow (Do Carmo (1992), Chapter 7), the geodesics can be extended indefinitely, (i.e., $a =-\infty$  and $b= \infty$), i.e., it is \emph{geodesically complete}, if and only if  $(M,\rho_g)$ is a complete metric space; this in turn is equivalent to the topological condition \eqref{eq-compact}. In particular, all compact Riemannian manifolds are \emph{geodesically complete}. In most of the applications in this article $M$ is compact.

On a complete Riemannian manifold, a geodesic $\gamma(t) = \gamma(t;p,v)$, $t\geq 0$, in the direction $v$, is completely determined by an initial point $p= \gamma(0)$, and an initial velocity $v=(d\gamma(t)/dt)t=0$ .  A \emph{cut point} of $p$ of the geodesic $\gamma$ along $v$ is $\gamma(r(v); p,v)$, where $r(v)$ is the supremum of all $t_0$ such that $\gamma$ is distance minimizing between $p=\gamma(0)$ and $\gamma(t_0)$. The set of all cut points (along all $v$) is called the \emph{cut locus} of $p$, denoted $\text{Cut}(p)$. The geodesic distance $q\rightarrow \rho_g(p,q)$ may not be smooth at the cut locus $\text{Cut}(p)$, as Example \ref{ex-2.2ap} below shows.  Next, define the \emph{exponential function} $Exp_p: T_p(M) \rightarrow M:  Exp_p (v) = \gamma(1;p,v)$ the point in $M$ reached by the geodesic in time $t=1$, starting at $p$ with an initial velocity $v$. It is known that $Exp_p$ is a diffeomorphism on an open ball $B(0:r_0)$ of $T_p(M)$, of radius $r_0=r_0(p)<\infty $, onto $M\backslash \text{Cut}(p)$ (\cite{docarmo}, p. 271).  Here $r_0 = r_0(p)$ is the geodesic distance between $p$ and $\text{Cut}(p)$). The inverse map $Exp_p^{-1} :M\backslash \text{Cut}(p) \rightarrow Exp_p (B(0: r_0)$ is called the \emph{inverse exponential}, or the \emph{$\log$ map}, $\log_p$, at $p$. The quantity $\text{inj}(M) = \sup\{r_0(p); p \in M)\} $ is the \emph{injectivity radius} of $M$. The $\log_p$ map also provides the so called \emph{normal coordinates} for a neighborhood of $p$.

\begin{example}[Exponential and Log Maps on the Sphere $S^d$ ]
\label{ex-2.2ap}

 Consider the unit sphere $S^d= \{x\in\R^{d+1}: |x|^2 \sum_{j=1}^d  (x^{(j)})^2 = 1\}$.  Because $|\gamma (t)| =1$ $\forall t$  for a curve on $S^d$, the tangent space at $p$  may be identified as the set of vectors in $\R^{d+1}$ orthogonal to $p$, $T_p(S^d) =   \{v \in \R^{d+1}: pv'=0\}$. Here we write $p$, $v$, etc. as row vectors. The geodesics are the big circles, so that the point reached at time one by the geodesic from $p$ moving with an initial velocity $v$ is the point on the big circle lying on the plane spanned by $p$ and $v$ at an arc distance $|v|$, i.e.,
  \begin{align}
   Exp_p(v) = \cos(|v|)p + \sin(|v|)v/|v|,\; v \neq 0,  Exp_p(0) = p\;  (pv'=0).                        
  \end{align}
Also, the geodesic distance between $p$ and $q$ is the smaller of the lengths $|v|$ of the two arcs joining $p$ and $q$ on the big circle,
 \begin{align}
 \rho_g(p,q) =  \arc \cos pq'   \in [0, \pi]. 
 \end{align}
Note that the cut locus of $p$ is $\text{Cut}(p) = \{-p\}$, and the distance between $p$ and $-p$ is $\pi$, and   $\text{inj}(S^d) = \pi$.  Hence the map $\log_p(q)$ is defined on $S^d\backslash\{-p\} $ and obtained by solving for $v $ the equation $exp_p(v) =q$.  Now $|v|= \rho_g(p,q)$. Plugging this in (A.1) (and using (A.2)), one has
$$\sqrt{[1 -  (pq')^2]} v = [q -  (pq')p](\arc\cos pq'), (p\neq q),$$
 which yields
 \begin{align}
 \label{eq-50}
 \log_p(q) &\nonumber=  [q  -  (pq')p](\arc\cos p'q) /\sqrt{[1  - (pq')2]}\\
& =  [\rho_g(p,q)/ \sin \rho_g(p,q)] [q  - (pq')p],                       
 \end{align}
 for $ q\neq p$, $q\neq- p$,    $\log_p(p) =0$.
The map $ \log_p(q)$ is a diffeomorphism   on $S^d\backslash\{-p\}$ onto  $\{v\in T_pS^d: |v| < \pi\}$.
If one uses complex coordinates for $p$, $q$ then $pq'$ in the formula above are to be replaced by $Re(pq^*)$, etc. 
\end{example}

Most of the manifolds we consider in this article are of the form $M=N/\mathcal G$. Here $N$ is a complete Riemannian manifold with a metric tensor $\rho_{g,N}$ and $\mathcal G$  is a compact Lie group of isometries \emph{acting freely} on $N$, i.e., except for the identity map, no $g$ in $\mathcal G$ has a fixed point. This means that the orbit $O_p$ of a point $p$ under $\mathcal G$ is in one-one correspondence with $\mathcal G$. As a subset of $N$, $O_p$ is a submanifold of $N$ of dimension that of $\mathcal G$. Its tangent space $T_p O_p$ as a subspace of $T_pN$  is called the \emph{vertical subspace} of $T_pN$, denoted $V_p$. The subspace $H_p$ of $T_pN$ orthogonal to $V_p$ is the \emph{horizontal subspace}.   $M$ is then a Riemannian manifold with the metric tensor.  The projection $\pi: N\rightarrow M$ is a \emph{Riemannian submersion} The quotient $N/\mathcal G$ is then a Riemannian manifold.

The final important notion from geometry needed in this section is that of curvature.  First, consider a smooth unit speed curve $\gamma$ in $\RR^2$: $1= |\dot\gamma(t) |^2 =\langle \dot\gamma(t), \dot\gamma(t) \rangle$. Differentiation shows that $\ddot\gamma (t) = d^2\gamma(t)/dt^2$ is orthogonal to  $\dot\gamma(t)  : \ddot\gamma (t) = \kappa(t)N(t)$, where $N(t)$ is a unit vector orthogonal to $\dot\gamma(t)$ such that $(\dot\gamma(t), N(t))$ has the same orientation as $(\partial/\partial x_1, \partial/\partial x_2)$. Then $\kappa(t)$ is the curvature of $\gamma$ at the point $\gamma(t)$.  Next, at a point $p $ on a regular surface $S$ in $\RR^3$, let $N= N(p)$ denote a unit normal  to $S$ at $p$. A plane $\pi$ through $N(p$) intersects $S$ in a smooth curve.  Let $\kappa(. ; p,\pi)$ be the curvature of this curve.  As $\pi$ varies by degrees of rotation, the curvature varies. Let $\kappa_1$ be the maximum and $\kappa_2$ the minimum of these curvatures, and let $\kappa = \kappa_1\kappa_2$ . The \emph{Theorem Egregium} of Gauss says that $\kappa  = \kappa(p)$ ($p \in S$), the so-called \emph{Gaussian curvature,}  is intrinsic to the surface $S$, i.e., it is the same for all surfaces isometric to $S$ (See, e.g., Boothby (1986), pp. 377-381). We now consider, somewhat informally, the case of a Riemannian manifold $M$.  For $p \in M$ and $u,v \in T_p(M)$, consider the two dimensional subspace $\pi$ spanned by $u$,$v$. Consider the two-dimensional submanifold swept out by geodesics in M with initial velocities lying in this subspace. The Gaussian curvature of this submanifold, thought of locally as a surface, is called the \emph{ sectional curvature} of $M$ at $p$ for the section $\pi$.

\section*{Acknowledgement}
The authors acknowledge support for this article from NSF grants DMS 1406872, CAREER 1654579 and IIS 1663870.

\bibliography{reference-career}
\bibliographystyle{apalike}

\end{document}